\documentclass[oneside,english]{amsart}
\usepackage[T1]{fontenc}
\usepackage{color}
\usepackage{array}
\usepackage{longtable}
\usepackage{amsthm}
\usepackage{graphicx}
\usepackage{amssymb}

\providecommand{\tabularnewline}{\\}

\numberwithin{equation}{section} 
\numberwithin{figure}{section} 
\theoremstyle{plain}
\newtheorem{thm}{Theorem}[section]
  \theoremstyle{definition}
  \newtheorem{defn}[thm]{Definition}
  \theoremstyle{plain}
  \newtheorem{prop}[thm]{Proposition}
  \theoremstyle{plain}
  \newtheorem{algorithm}[thm]{Algorithm}
  \theoremstyle{plain}
  
 \theoremstyle{definition}
  \newtheorem{example}[thm]{Example}
  \theoremstyle{plain}
  \newtheorem{assumption}[thm]{Assumption}
  \theoremstyle{plain}
  
  \theoremstyle{plain}
  \newtheorem{lem}[thm]{Lemma}
  \theoremstyle{remark}
  \newtheorem{rem}[thm]{Remark}

\usepackage[small,nohug,heads=littlevee]{diagrams}
\diagramstyle[labelstyle=\scriptstyle]

\newcommand{\lev}{\mbox{\rm lev}}

\newcommand{\cl}{\mbox{\rm cl}}
\newcommand{\conv}{\mbox{\rm conv}}

\newcommand{\wlim}{\mathop{\rm{w-lim}}}
\newcommand{\dom}{\mbox{\rm dom}}

\title{Level set methods for finding critical points of mountain pass type}

\begin{document}

\date{\today}

\author{A.S. Lewis }

\curraddr{School of Operations Research and Information Engineering, Cornell
University, Ithaca, NY 14853. }

\email{aslewis@orie.cornell.edu.}

\author{C.H.J. Pang}

\curraddr{Combinatorics and Optimization, University of Waterloo, 200 University
Ave W., Waterloo, ON, Canada N2l 3G1.}

\email{chj2pang@math.uwaterloo.ca}
\begin{abstract}
Computing mountain passes is a standard way of finding critical points.
We describe a numerical method for finding critical points that is
convergent in the nonsmooth case and locally superlinearly convergent
in the smooth finite dimensional case. We apply these techniques to
describe a strategy for the Wilkinson problem of calculating the distance
of a matrix to a closest matrix with repeated eigenvalues. Finally,
we relate critical points of mountain pass type to nonsmooth and metric
critical point theory.
\end{abstract}
\maketitle
\tableofcontents{}


\section{Introduction}

Computing mountain passes is an important problem in computational
chemistry and in the study of nonlinear partial differential equations.
We begin with the following definition.
\begin{defn}
\label{def:mountain-pass}Let $X$ be a topological space, and consider
$a,b\in X$. For a function $f:X\rightarrow\mathbb{R}$, define a
\emph{mountain pass} $p^{*}\in\Gamma(a,b)$ to be a minimizer of the
problem \[
\inf_{p\in\Gamma(a,b)}\sup_{0\leq t\leq1}f\circ p(t).\]
Here, $\Gamma(a,b)$ is the set of continuous paths $p:[0,1]\rightarrow X$
such that $p(0)=a$ and $p(1)=b$. 
\end{defn}
An important problem in computational chemistry is to find the lowest
energy to transition between two stable states. If $a$ and $b$ represent
two states and $f$ maps the states to their potential energies, then
the mountain pass problem calculates this lowest energy. Early work
on computing transition states includes Sinclair and Fletcher \cite{SF74},
and recent work is reviewed by Henkelman, J\'{o}hannesson and J\'{o}nsson
\cite{HJJ00}. We refer to this paper for further references in the
Computational Chemistry literature. 

Perhaps more importantly, the mountain pass idea is also a useful
tool in the analysis of nonlinear partial differential equations.
For a Banach space $X$, variational problems are problems (P) such
that there exists a smooth functional $J:X\rightarrow\mathbb{R}$
whose critical points (points where $\nabla J=0$) are solutions of
(P). Many partial differential equations are variational problems,
and critical points of $J$ are {}``weak'' solutions. In the landmark
paper by Ambrosetti and Rabinowitz \cite{AR73}, the mountain pass
theorem gives a sufficient condition for the existence of critical
points in infinite dimensional spaces. If an optimal path to solve
the mountain pass problem exists and the maximum along the path is
greater than $\max(f(a),f(b))$, then the maximizer on the path is
a critical point distinct from $a$ and $b$. The mountain pass theorem
and its variants are the primary ways to establish the existence of
critical points and to find critical points numerically. For more
on the mountain pass theorem and some of its generalizations, we refer
the reader to \cite{Jabri03}. 

In \cite{CM93}, Choi and McKenna proposed a numerical algorithm for
the mountain pass problem by using an idea from Aubin and Ekeland
\cite{AE84} to solve a semilinear partial differential equation.
This is extended to find solutions of \emph{Morse index} 2 (that is,
the maximum dimension of the subspace of $X$ on which $J^{\prime\prime}$
is negative definite) in Ding, Costa and Chen \cite{DCC99}, and then
to higher Morse index by Li and Zhou \cite{LZ01}. 

Li and Zhou \cite{LZ02}, and Yao and Zhou \cite{YZ07} proved convergence
results to show that their minimax method is sound for obtaining weak
solutions to nonlinear partial differential equations. Mor\'{e} and
Munson \cite{MM04} proposed an {}``elastic string method'', and
proved that the sequence of paths created by the elastic string method
contains a limit point that is a critical point.

The prevailing methods for numerically solving the mountain pass problem
are motivated by finding a sequence of paths (by discretization or
otherwise) such that the maximum along these paths decrease to the
optimal value. Indeed, many methods in \cite{HJJ00} approximate a
mountain pass in this manner. As far as we are aware, only \cite{BT07,H04}
deviate from this strategy. We make use of a different approach by
looking at the path connected components of the lower level sets of
$f$  instead.

One easily sees that $l$ is a lower bound of the mountain pass problem
if and only if $a$ and $b$ lie in two different path connected components
of $\lev_{\leq l}f$. A strategy to find an optimal mountain pass
is to start with a lower bound $l$ and keep increasing $l$ until
the path connected components of $\lev_{\leq l}f$ containing $a$
and $b$ respectively coalesce at some point. However, this strategy
requires one to determine whether the points $a$ and $b$ lie in
the same path connected component, which is not easy. We turn to finding
saddle points of mountain pass type, as defined below.
\begin{defn}
\label{def:saddle-point}For a function $f:X\rightarrow\mathbb{R}$,
a \emph{saddle point of mountain pass type} $\bar{x}\in X$ is a point
such that there exists an open set $U$ such that $\bar{x}$ lies
in the closure of two path components of $(\lev_{<f(\bar{x})}f)\cap U$.
\end{defn}
We shall refer to saddle points of mountain pass type simply as saddle
points. As an example, for the function $f:\mathbb{R}^{2}\rightarrow\mathbb{R}$
defined by $f(x)=x_{1}^{2}-x_{2}^{2}$, the point $\mathbf{0}$ is
a saddle point of mountain pass type: We can choose $U=\mathbb{R}^{2}$,
$a=(0,1)$, $b=(0,-1)$. When $f$ is $\mathcal{C}^{1}$, it is clear
that saddle points are critical points. As we shall see later (in
Propositions \ref{pro:equiv-mountain} and \ref{pro:semi-algebraic-mtn}),
saddle points of mountain pass type can, under reasonable conditions,
be characterized as maximal points on mountain passes, acting as {}``bottlenecks''
between two components. In fact, if $f$ is $\mathcal{C}^{2}$, the
Hessians are nonsingular and several mild assumptions hold, these
bottlenecks are exactly critical points of Morse index 1. We refer
the reader to the lecture notes by Ambrosetti \cite{Amb92}. Some
of the methods in \cite{HJJ00} actually find saddle points instead
of solving the mountain pass problem. 

We propose numerical methods to find saddle points using the strategy
suggested in Definition \ref{def:saddle-point}. We start with a lower
bound $l$ and keep increasing $l$ until the components of the level
set $\lev_{\leq l}f\cap U$ containing $a$ and $b$ respectively
coalesce, reaching the objective of the mountain pass problem. The
first method we propose in Algorithm \ref{alg:globally-convergent-MPT}
is purely metric in nature. One appealing property of this method
is that calculations are now localized near the critical point and
we keep track of only two points instead of an entire path. Our algorithm
enjoys a monotonicity property: The distance between two components
decreases monotonically as the algorithm progresses, giving an indication
of how close we are to the saddle point. In a practical implementation,
local optimality properties in terms of the gradients (or generalized
gradients) can be helpful for finding saddle points. Such optimality
conditions are covered in Section \ref{sec:optim-conditions}.

It follows from the definitions that our algorithm, if it converges,
converges to a saddle point. We then prove that any saddle point is
deformationally critical in the sense of metric critical point theory
\cite{DM94,Katriel94,IS96}, and is Morse critical under additional
conditions. This implies in particular that any saddle point is Clarke
critical in the sense of nonsmooth critical point theory \cite{Chang81,Shi85}
based on nonsmooth analysis in the spirit of \cite{BZ05,Cla83,Mor06,RW98}.
It seems that there are few existing numerical methods for finding
either critical points in a metric space or nonsmooth critical points.
Currently, we are only aware of \cite{YZ05}.   

One of the main contributions of this paper is to give a second method
(in Section \ref{sec:locally-superlinearly-convergent}) which converges
locally superlinearly to a nondegenerate smooth critical point, i.e.,
critical points where the Hessian is nonsingular, in $\mathbb{R}^{n}$.
A potentially difficult step in this second method is that we have
to find the closest point between two components of the level sets.
While the effort meeded to perform this step accurately may be great,
the purpose of this step is to make sure that the problem is well
aligned after this step. Moreover, this step need not be performed
to optimality. In our numerical example in Section \ref{sec:Wilkinson-2},
we were able to obtain favorable results without performing this step.

Our initial interest in the mountain pass problem came from computing
the $2$-norm distance of a matrix $A$ to the closest matrix with
repeated eigenvalues. This is also known as the Wilkinson problem,
and this value is the smallest $2$-norm perturbation that will make
the eigenvalues of matrix $A$ behave in a non-Lipschitz manner. Alam
and Bora \cite{AB05a} showed how the Wilkinson's problem can be reduced
to a global mountain pass problem. We do not solve the global mountain
pass problem associated with the Wilkinson problem, but we demonstrate
that locally our algorithm converges quickly to a smooth critical
point of mountain pass type.

\textbf{Outline: }Section \ref{sec:global-convergent} illustrates
a local algorithm to find saddle points of mountain pass type, while
Sections \ref{sec:locally-superlinearly-convergent}, \ref{sec:Proof-of-superlinear}
and \ref{sec:fast-local-observations} are devoted to the statement,
proof of convergence, and additional observations of a fast local
algorithm to find nondegenerate critical points of Morse index 1 in
$\mathbb{R}^{n}$. 

Sections \ref{sec:Global-convergence} discusses the relationship
between mountain passes, saddle points, and critical points in the
sense of metric critcal point theory and nonsmooth analysis, and does
not depend on material in Sections \ref{sec:locally-superlinearly-convergent},
\ref{sec:Proof-of-superlinear} and \ref{sec:fast-local-observations}. 

Finally, Sections \ref{sec:Wilkinson-1} and \ref{sec:Wilkinson-2}
illustrates the fast local algorithm in Section \ref{sec:locally-superlinearly-convergent}.
Section \ref{sec:optim-conditions} discusses optimality conditions
for the subproblem in the algorithm in Section \ref{sec:global-convergent}. 

\textbf{Notation:} As we will encounter situations where we want to
find the square of the $j$th coordinate of the $i$th iterate of
$x$, we write $x_{i}^{2}(j)$ in the proof of Theorem \ref{thm:superline-conv}.
In other parts, it will be clear from context whether the $i$ in
$x_{i}$ is used as an iteration counter or as a reference to the
$i$th coordinate.  Let $\mathbb{B}^{d}(\mathbf{0},r)$ be the ball
with center $\mathbf{0}$ and radius $r$ in $\mathbb{R}^{d}$, and
$\mathring{\mathbb{B}}^{d}(\mathbf{0},r)$ be the corresponding open
ball.

\section{\label{sec:global-convergent}A level set algorithm}

We present a level set algorithm to find saddle points.  Assume
$f:X\rightarrow\mathbb{R}$, where $(X,d)$ is a metric space.
\begin{algorithm}
(Level set algorithm) \label{alg:globally-convergent-MPT}A local
bisection method for approximating a mountain pass from $x_{0}$ to
$y_{0}$ for $f\mid_{U}$, where both $x_{0}$ and $y_{0}$ lie in
some open path connected set $U$.\end{algorithm}
\begin{enumerate}
\item Start with an upper bound $u$ and a lower bound $l$ for the objective
of the mountain pass problem and $i=0$.
\item Solve the optimization problem \begin{eqnarray}
 & \min & d(x,y)\nonumber \\
 & \mbox{s.t.} & x\in S_{1},y\in S_{2}\label{eq:optimization}\end{eqnarray}
where $S_{1}$ is the component of the level set $(\lev_{\leq\frac{1}{2}(l+u)}f)\cap U$
that contains $x_{i}$ and $S_{2}$ is the component that contains
$y_{i}$. 
\item If $S_{1}$ and $S_{2}$ are the same component, then $\frac{1}{2}(l+u)$
is an upper bound, otherwise it is a lower bound. Update the upper
and lower bounds accordingly. In the case where the lower bound is
changed, increase $i$ by $1$, and let $x_{i}$ and $y_{i}$ be the
minimizers of \eqref{eq:optimization}. For future discussions, let
$l_{i}$ corresponding value of $l$ to $x_{i}$ and $y_{i}$. Repeat
step 2 until $x_{i}$ and $y_{i}$ are sufficiently close.
\item If an actual approximate mountain pass is desired, take a path $p_{i}:[0,1]\rightarrow U\cap(\lev_{\leq u}f)$
connecting the points\[
x_{0},x_{1},\dots,x_{i-2},x_{i-1},x_{i},y_{i},y_{i-1},y_{i-2},\dots,y_{1},y_{0}.\]

\end{enumerate}
Step (3) is illustrated in Figure \ref{fig:Illustrate-global}.

\begin{figure}
\begin{longtable}{|>{\raggedright}b{1.1in}|c|c|}
\hline 
Case & Before & After\tabularnewline
\hline
\noindent $\left\{ x\mid f(x)\leq\frac{u+l}{2}\right\} $

\noindent 2 components & \begin{tabular}{c}
\includegraphics[scale=0.3]{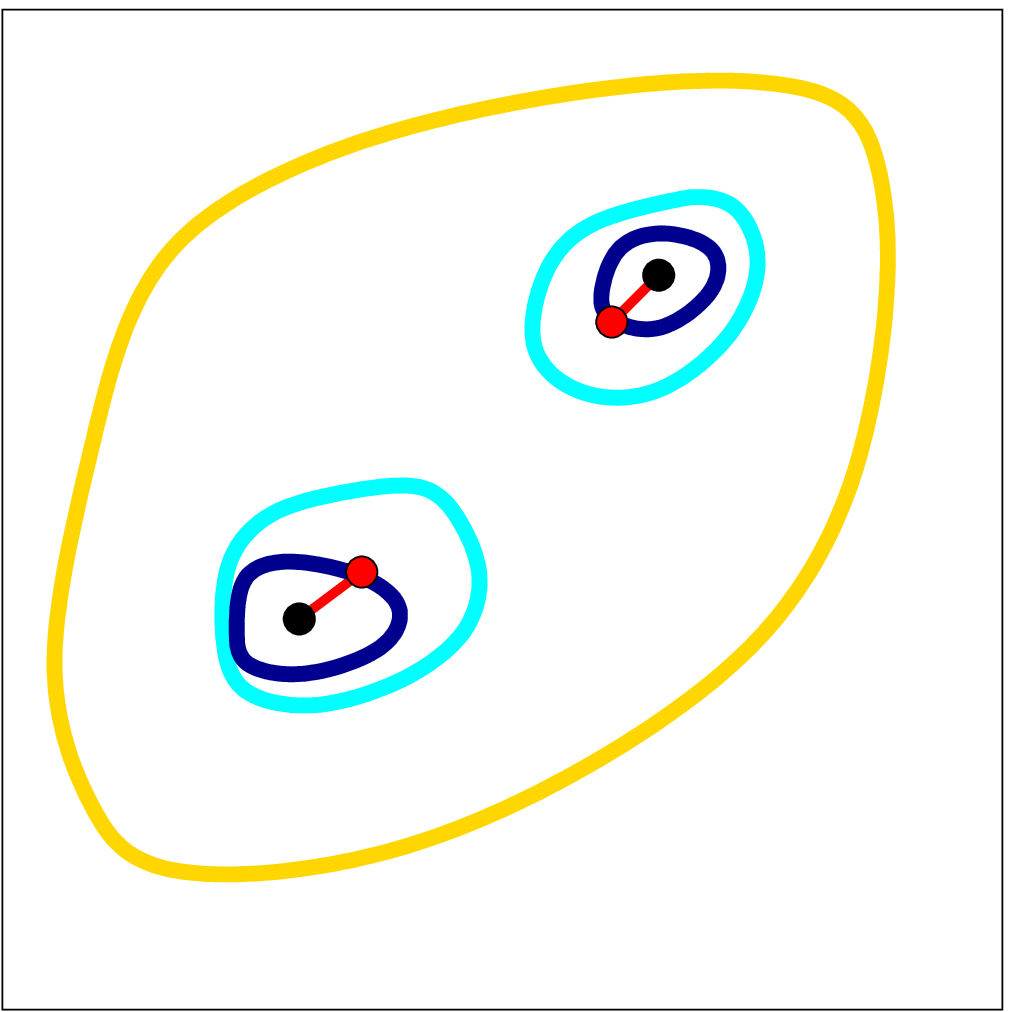}\tabularnewline
\end{tabular} & \begin{tabular}{c}
\includegraphics[scale=0.3]{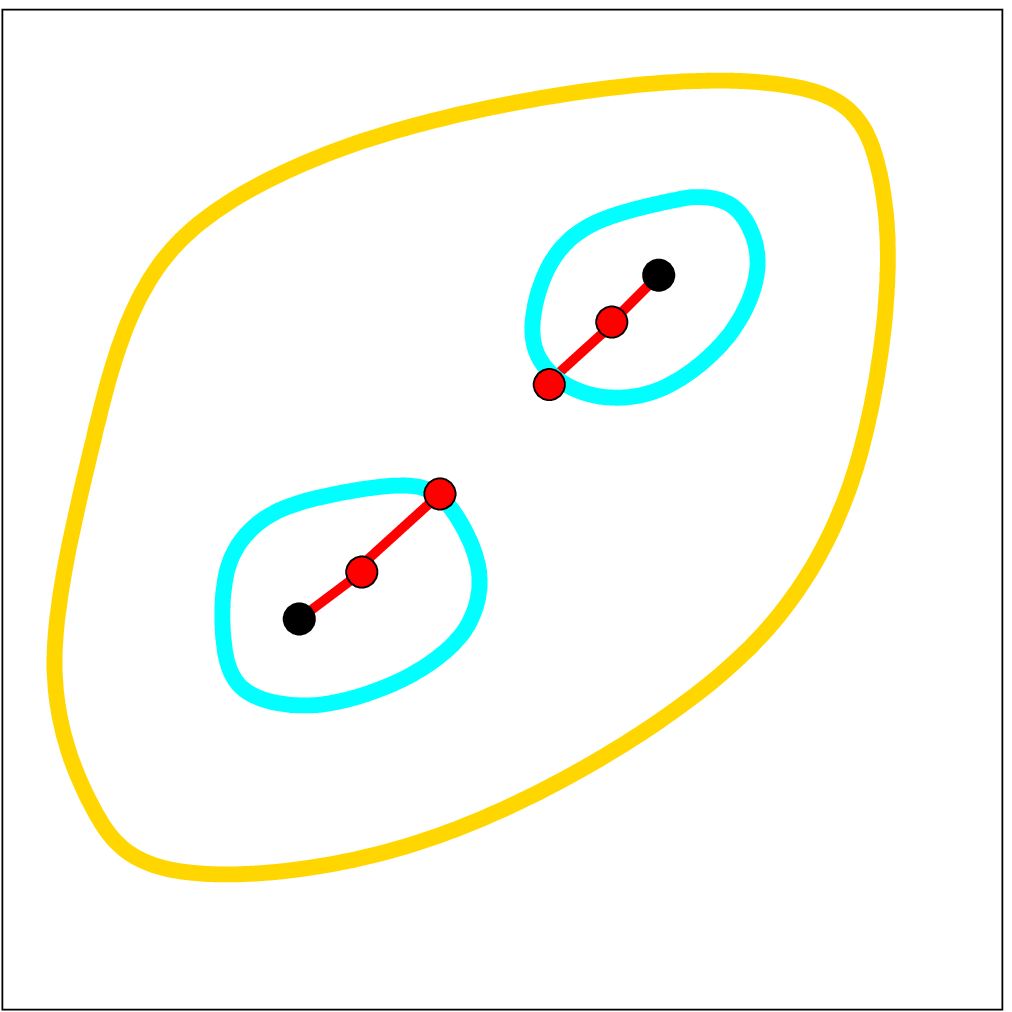}\tabularnewline
\end{tabular}\tabularnewline
\hline
1 component & \begin{tabular}{c}
\includegraphics[scale=0.3]{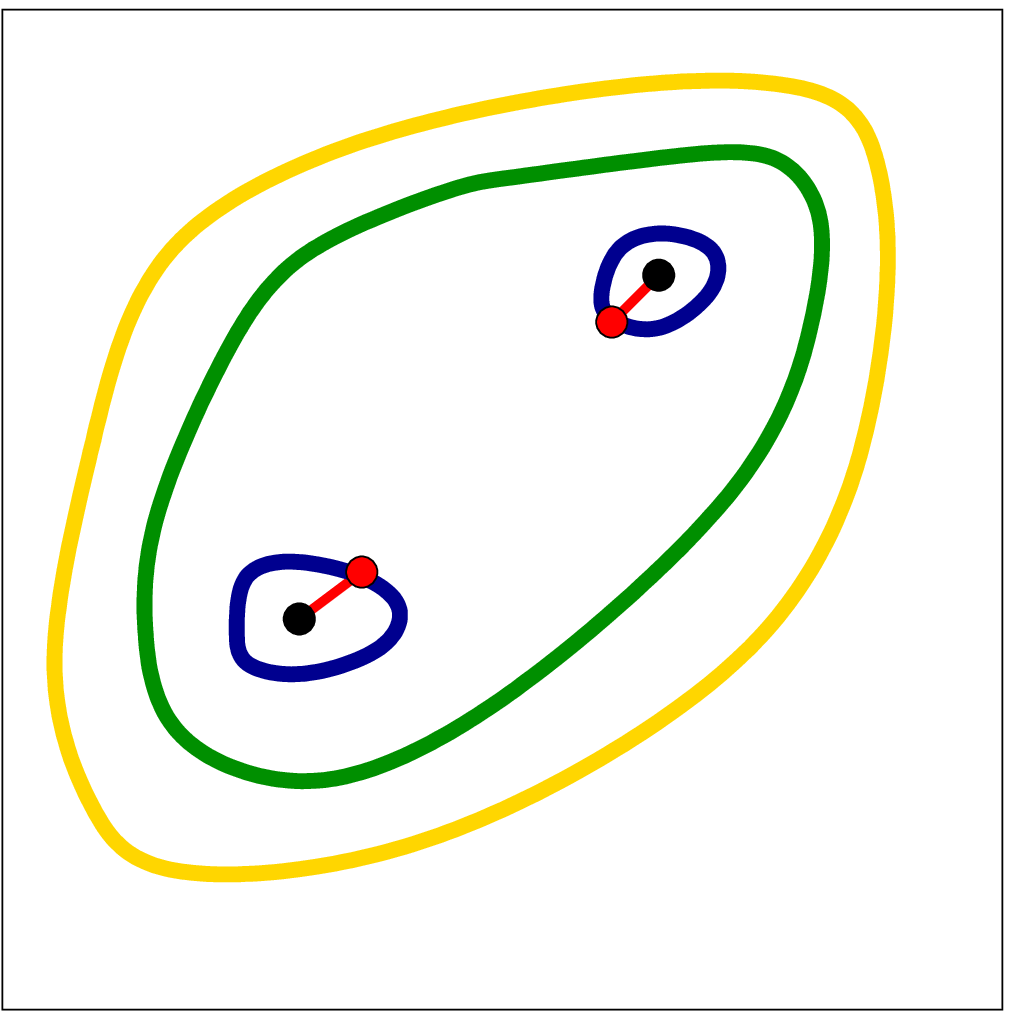}\tabularnewline
\end{tabular} & \begin{tabular}{c}
\includegraphics[scale=0.3]{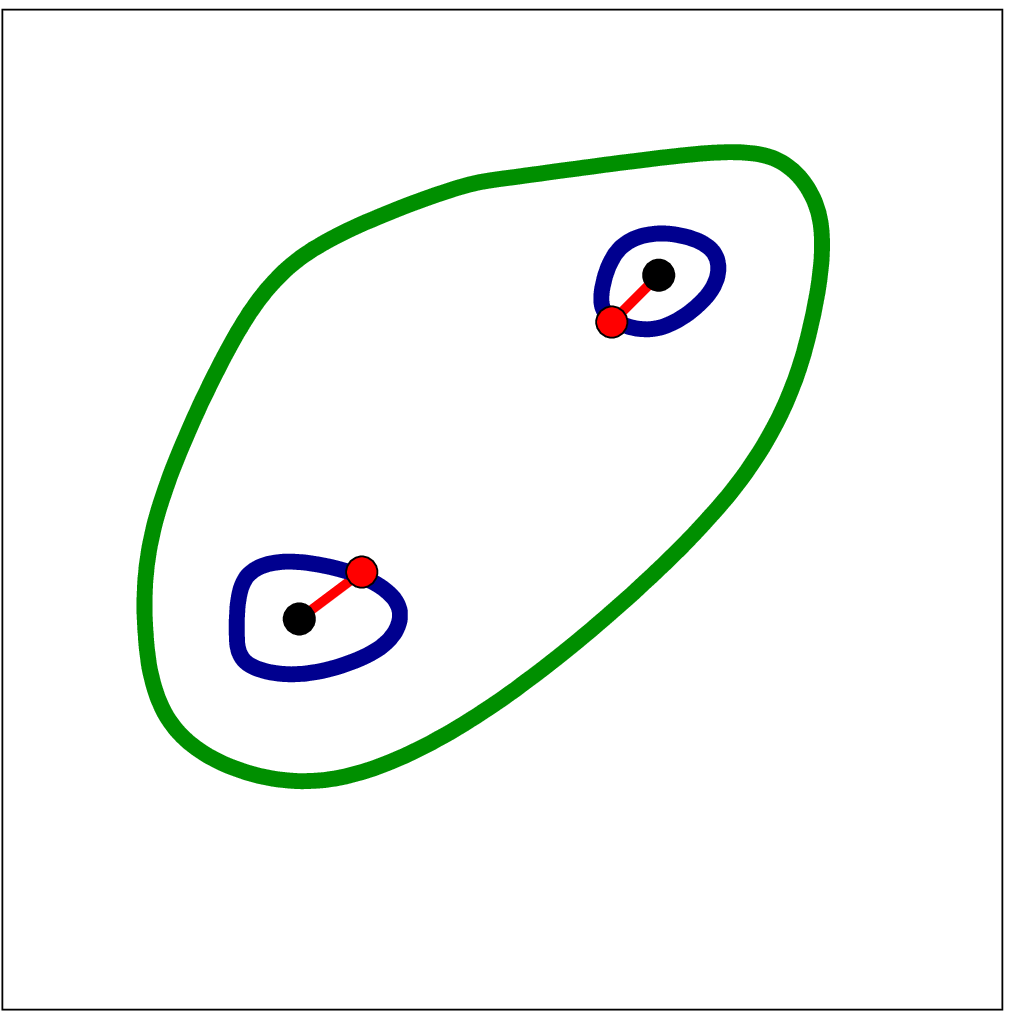}\tabularnewline
\end{tabular}\tabularnewline
\hline
\end{longtable}\caption{\label{fig:Illustrate-global}Illustration of Algorithm \ref{alg:globally-convergent-MPT}.}

\end{figure}

 To start the algorithm, an upper bound $u$ can be taken to be
the maximum of any path from $x_{0}$ to $y_{0}$, while a lower bound
can be the maximum of $f(x_{0})$ and $f(y_{0})$. In fact, in step
(3), we may update the upper bound $u$ to be the maximum along the
line segment joining $x_{i}$ and $y_{i}$ if it is a better upper
bound.

In practice, one need not solve subproblem \eqref{eq:optimization}
in step 2 too accurately, as it might be more profitable to move on
to step 3.  While theory demands the global optimizers for subproblem
\eqref{eq:optimization}, an implementation of Algorithm \ref{alg:globally-convergent-MPT}
can only find local optimizers, which is not sufficient for the global
mountain pass problem, but can be successful for the purpose of finding
saddle points. The optimality conditions in terms of gradients (or
generalized gradients) can be helpful for characterizing local optimality
(see Section \ref{sec:optim-conditions}). Notice that the saddle
point property is local. If $x_{i}$ and $y_{i}$ converge to a common
limit, then it is clear from the definitions that the common limit
is a saddle point.

Another issue with subproblem \eqref{eq:optimization} in step 2 is
that minimizers may not exist. For example, the sets $S_{1}$ and
$S_{2}$ may not be compact. We now discuss how convergence to a critical
point in Algorithm \ref{alg:globally-convergent-MPT} can fail in
the finite dimensional case. 

The Palais-Smale condition is important in nonlinear analysis, and
is often a necessary condition in the smooth and nonsmooth mountain
pass theorems and other critical point existence theorems. We refer
to \cite{MW89,Nir89,Rab86,Struwe00,Wil92} for more details. We recall
its definition.
\begin{defn}
Let $X$ be a Banach space and $f:X\rightarrow\mathbb{R}$ be a $\mathcal{C}^{1}$
functional. We say that a sequence $\left\{ x_{i}\right\} _{i=1}^{\infty}\subset X$
is a \emph{Palais-Smale sequence} if $\left\{ f(x_{i})\right\} _{i=1}^{\infty}$
is bounded and $f^{\prime}(x_{i})\rightarrow\mathbf{0}$, and $f$
satisfies the \emph{Palais-Smale condition} if any Palais-Smale sequence
admits a convergent subsequence.

For nonsmooth $f$, the condition $f^{\prime}(x_{i})\rightarrow\mathbf{0}$
is $\inf_{x_{i}^{*}\in\partial f(x_{i})}\left|x_{i}^{*}\right|\rightarrow0$
instead.
\end{defn}
In the absence of the Palais-Smale condition, Algorithm \ref{alg:globally-convergent-MPT}
may fail to converge because the sequence $\left\{ (x_{i},y_{i})\right\} _{i=1}^{\infty}$
need not have a limit point of the form $\left(\bar{z},\bar{z}\right)$,
or the sequence $\left\{ (x_{i},y_{i})\right\} _{i=1}^{\infty}$ need
not even exist. The examples below document the possibilities. 
\begin{example}
(a) Consider $f:\mathbb{R}^{2}\rightarrow\mathbb{R}$ defined by $f(x,y)=e^{-x}-y^{2}$.
Here, the distance between the two components of the level sets is
zero for all $\mbox{lev}_{\leq c}f$, where $c<0$, and $x_{i}$ and
$y_{i}$ do not exist. The sequence $\left\{ (i,0)\right\} _{i=1}^{\infty}$
is a Palais-Smale sequence but does not converge.

(b) For $f(x,y)=e^{-2x}-y^{2}e^{-x}$ , $x_{i}$ and $y_{i}$ exist,
but both $\left\{ x_{i}\right\} _{i=1}^{\infty}$ and $\left\{ y_{i}\right\} _{i=1}^{\infty}$
do not have finite limits. Again, $\left\{ (i,0)\right\} _{i=1}^{\infty}$
is a Palais-Smale sequence that does not converge.
\end{example}
It is possible that $\left\{ x_{i}\right\} _{i=1}^{\infty}$ and $\left\{ y_{i}\right\} _{i=1}^{\infty}$
have limit points but not a common limit point. To see this, consider
the example $f:\mathbb{R}\rightarrow\mathbb{R}$ defined by\[
f(x)=\left\{ \begin{array}{ll}
x & \mbox{ if }x\leq-1\\
-1 & \mbox{ if }-1\leq x\leq1\\
-x & \mbox{ if }x\geq1.\end{array}\right.\]
The set $\mbox{lev}_{\leq-1}f$ is path-connected, but the set $\mbox{cl}(\mbox{lev}_{<-1}f)$
is not path-connected. Any point in the set $(\mbox{lev}_{\leq-1}f)\backslash\mbox{cl}(\mbox{lev}_{<-1}f)=(-1,1)$
is a local minimum, and hence a critical point.

\section{\label{sec:locally-superlinearly-convergent}A locally superlinearly
convergent algorithm}

In this section, we propose a locally superlinearly convergent algorithm
for the mountain pass problem for smooth critical points in $\mathbb{R}^{n}$.
For this section, we take $X=\mathbb{R}^{n}$. Like Algorithm \ref{alg:globally-convergent-MPT}
earlier, we keep track of only two points in the space $\mathbb{R}^{n}$
instead of a path. Our fast locally convergent algorithm does not
require one to calculate the Hessian. Furthermore, we maintain upper
and lower bounds that converge superlinearly to the critical value.
The numerical performance of this method will be illustrated in Section
\ref{sec:Wilkinson-2}.

In Algorithm \ref{alg:(Mountain-pass-1)} below, we can assume that
the endpoints $x_{0}$ and $y_{0}$ satisfy $f(x_{0})=f(y_{0})$.
Otherwise, if $f(x_{0})<f(y_{0})$ say, replace $x_{0}$ by the point
$x_{0}^{\prime}$ closest to $x_{0}$ on the line segment $[x_{0},y_{0}]$
such that $f(x_{0}^{\prime})=f(y_{0})$. 
\begin{algorithm}
(Fast local level set algorithm) \label{alg:(Mountain-pass-1)} Find
saddle point between points $x_{0}$ and $y_{0}$ for $f:\mathbb{R}^{n}\rightarrow\mathbb{R}$.
Assume that the objective of the mountain pass problem between $x_{0}$
and $y_{0}$ is greater than $f(x_{0})$, and $f(x_{0})=f(y_{0})$.
Let $U$ be a convex set containing $x_{0}$ and $y_{0}$.\end{algorithm}
\begin{enumerate}
\item Given points $x_{i}$ and $y_{i}$, find $z_{i}$ as follows:

\begin{enumerate}
\item Replace $x_{i}$ and $y_{i}$ by $\tilde{x}_{i}$ and $\tilde{y}_{i}$,
where $\tilde{x}_{i}$ and $\tilde{y}_{i}$ are minimizers of the
problem\begin{eqnarray*}
 & \min_{x,y} & \left|x-y\right|\\
 & \mbox{s.t.} & x\mbox{ in same component as }x_{i}\mbox{ in }(\lev_{\leq f(x_{i})}f)\cap U\\
 &  & y\mbox{ in same component as }y_{i}\mbox{ in }(\lev_{\leq f(x_{i})}f)\cap U\end{eqnarray*}

\item Find a minimizer of $f$ on $L_{i}\cap U$, say $z_{i}$. Here $L_{i}$
is the affine space orthogonal to $x_{i}-y_{i}$ passing through $\frac{1}{2}(x_{i}+y_{i})$.

\end{enumerate}
\item Find the point furthest away from $x_{i}$ on the line segment $[x_{i},z_{i}]$,
which we call $x_{i+1}$, such that $f(x)\leq f(z_{i})$ for all $x$
in the line segment $[x_{i},x_{i+1}]$. Do the same to find $y_{i+1}$.
\item Increase $i$, repeat steps 1 and 2 until $\left|x_{i}-y_{i}\right|$
is small, or if the value $M_{i}-f(z_{i})$, where $M_{i}:=\max_{x\in[x_{i},y_{i}]}f(x),$
is small.
\item If an actual path is desired, take a path $p_{i}:[0,1]\rightarrow X$
lying in $\lev_{\leq M_{i}}f$ connecting the points\[
x_{0},x_{1},\dots,x_{i-2},x_{i-1},x_{i},y_{i},y_{i-1},y_{i-2},\dots,y_{1},y_{0}.\]

\end{enumerate}
As we will see in Propositions \ref{pro:alg-prelim} and \ref{pro:2-convex-sets},
a unique minimizing pair $(\tilde{x}_{i},\tilde{y}_{i})$ in step
1(a) exists under added conditions. Furthermore, Proposition \ref{eq:alpha-formula}
implies that a unique minimizer of $f$ on $L_{i}\cap U$ exists under
added conditions in step 1(b).

To motivate step 1(b), consider any path from $x_{i}$ to $y_{i}$
in $U$ that lies wholly in $U$. Such a path has to pass through
some point of $L_{i}\cap U$, so the maximum value of $f$ on the
path is at least the minimum of $f$ on $L_{i}\cap U$. 

Step 1(a) is analogous to step 2 of Algorithm \ref{alg:globally-convergent-MPT}.
Algorithm \ref{alg:(Mountain-pass-1)} can be seen as an improvement
Algorithm \ref{alg:globally-convergent-MPT}: The bisection algorithm
in Algorithm \ref{alg:globally-convergent-MPT} gives us a reliable
way of finding the critical point, and step 1(b) in Algorithm \ref{alg:(Mountain-pass-1)}
reduces the distance between the components of the level sets as fast
as possible.

In practice, step 1(a) is difficult, and is performed only when the
algorithm runs into difficulties. In fact, this step was not performed
in our numerical experiments in Section \ref{sec:Wilkinson-2}. However,
we can construct simple functions for which the affine space $L_{i}$
does not separate the two components containing $x_{i}$ and $y_{i}$
in $(\lev_{\leq f(x_{i})}f)\cap U$ in step 1(b) if step 1(a) were
not performed. 

In the minimum distance problem in step 1(a), notice that if $f$
is $\mathcal{C}^{1}$ and the gradients of $f$ at a pair of points
are nonzero and do not point in opposite directions, then in principle
we can perturb the points along paths that decrease the distance between
them while not increasing their function values. Of course, a good
approximation of a minimizing pair may be hard to compute in practice:
existing path-based algorithms for finding mountain passes face analogous
computational challenges. One may employ the heuristic in Remark \ref{rem:heuristic-closest-point}
for this problem.

In step 2, continuity of $f$ and $p$ tells us that $f(x_{i+1})=f(z_{i})$.
We shall see in Theorem \ref{thm:superline-conv} that under added
conditions, $\{f(x_{i})\}_{i}$ is an increasing sequence that converges
to the critical value $f(\bar{x})$. Furthermore, Propositions \ref{pro:alpha-bound}
and \ref{pro:upper-bound-better} state that under added conditions,
$\{M_{i}\}_{i}$ are upper bounds on $f(\bar{x})$ that converge R-superlinearly
to $f(\bar{x})$, where R-superlinear convergence is defined as follows.
\begin{defn}
A sequence in $\mathbb{R}$ converges \emph{R-superlinearly }to zero
if its absolute value is bounded by a superlinearly convergent sequence.

\end{defn}

\section{\label{sec:Proof-of-superlinear}Superlinear convergence of the local
algorithm}

When $f:\mathbb{R}^{n}\rightarrow\mathbb{R}$ is a quadratic whose
Hessian has one negative eigenvalue and $n-1$ positive eigenvalues,
Algorithm \ref{alg:(Mountain-pass-1)} converges to the critical point
in one step. One might expect that if $f$ is $\mathcal{C}^{2}$,
then Algorithm \ref{alg:(Mountain-pass-1)} converges quickly. In
this section, we will prove Theorem \ref{thm:superline-conv} on the
superlinear convergence of Algorithm \ref{alg:(Mountain-pass-1)}. 

Recall that the \emph{Morse index} of a critical point is the maximum
dimension of a subspace on which the Hessian is negative definite,
and a critical point is \emph{nondegenerate} if its Hessian is invertible,
and degenerate otherwise. In the smooth finite dimensional case, the
Morse index equals the number of negative eigenvalues of the Hessian.
If a function $f:\mathbb{R}^{n}\rightarrow\mathbb{R}$ is $\mathcal{C}^{2}$
in a neighborhood of a nondegenerate critical point $\bar{x}$ of
Morse index 1, we can readily make the following assumptions.
\begin{assumption}
\label{ass:simplify-quad}Assume that $\bar{x}=\mathbf{0}$ and $f(\mathbf{0})=0$,
and the Hessian $H=H(\mathbf{0})$ is a diagonal matrix with entries
$a_{1},a_{2},\dots,a_{n-1},a_{n}$ in decreasing order, of which $a_{n}$
is negative and $a_{n-1}$ is the smallest positive eigenvalue. 
\end{assumption}

Another assumption that we will use quite often in this section and
the next is on the local approximation of $f$ near $\mathbf{0}$.
\begin{assumption}
\label{ass:simplify-theta}For $\delta\in(0,\min\{a_{n-1},-a_{n}\})$,
assume $\theta>0$ is small enough so that \[
\left|f(x)-\sum_{j=1}^{n}a_{j}x^{2}(j)\right|\leq\delta\left|x\right|^{2}\mbox{ for all }x\in\mathbb{B}(\mathbf{0},\theta).\]

\end{assumption}
This particular choice of $\theta$ gives a region $\mathbb{B}(\mathbf{0},\theta)$
where Figure \ref{fig:Local-saddle} is valid. We shall use $\mathring{\mathbb{B}}$
to denote the open ball.

Here is our first result on step 1(a) of Algorithm \ref{alg:(Mountain-pass-1)}.
\begin{prop}
\label{pro:alg-prelim}Suppose that $f:\mathbb{R}^{n}\rightarrow\mathbb{R}$
is $\mathcal{C}^{2}$, and $\bar{x}$ is a nondegenerate critical
point of Morse index 1 such that $f(\bar{x})=c$. If $\theta>0$ is
sufficiently small, then for any $\epsilon>0$ (depending on $\theta$)
sufficiently small,
\begin{enumerate}
\item $(\lev_{\leq c-\epsilon}f)\cap\mathring{\mathbb{B}}(\bar{x},\theta)$
has exactly two path connected components, and 
\item There is a pair $(\tilde{x},\tilde{y})$, where $\tilde{x}$ and $\tilde{y}$
lie in distinct components of $(\lev_{\leq c-\epsilon}f)\cap\mathring{\mathbb{B}}(\bar{x},\theta)$,
such that $\left|\tilde{x}-\tilde{y}\right|$ is the distance between
the two components in $(\lev_{\leq c-\epsilon}f)\cap\mathring{\mathbb{B}}(\bar{x},\theta)$.
\end{enumerate}
\end{prop}
\begin{proof}
Suppose that Assumption \ref{ass:simplify-quad} holds. Choose some
$\delta\in(0,\min\{a_{n-1},-a_{n}\})$ and a corresponding $\theta>0$
such that Assumption \ref{ass:simplify-theta} holds. A simple bound
on $f(x)$ on $\mathbb{B}(\mathbf{0},\theta)$ is therefore:\begin{equation}
\sum_{j=1}^{n}(a_{j}-\delta)x^{2}(j)\leq f(x)\leq\sum_{j=1}^{n}(a_{j}+\delta)x^{2}(j).\label{eq:approx-label}\end{equation}
So if $\epsilon$ is small enough, the level set $S:=\lev_{\leq-\epsilon}f$
satisfies \[
S_{+}\cap\mathbb{B}(\mathbf{0},\theta)\subset S\cap\mathbb{B}(\mathbf{0},\theta)\subset S_{-}\cap\mathbb{B}(\mathbf{0},\theta),\]
where \begin{eqnarray*}
S_{+} & := & \left\{ x\mid\sum_{j=1}^{n}(a_{j}+\delta)x^{2}(j)\leq-\epsilon\right\} ,\\
S_{-} & := & \left\{ x\mid\sum_{j=1}^{n}(a_{j}-\delta)x^{2}(j)\leq-\epsilon\right\} ,\end{eqnarray*}
 and $S_{+}\cap\mathbb{B}(\mathbf{0},\theta)$ is nonempty. Figure
\ref{fig:Local-saddle} shows a two-dimensional cross section of the
sets $S_{+}$ and $S_{-}$ through the critical point $\mathbf{0}$
and the closest points between components in $S_{+}$ and $S_{-}$. 

\begin{figure}
\includegraphics[scale=0.5]{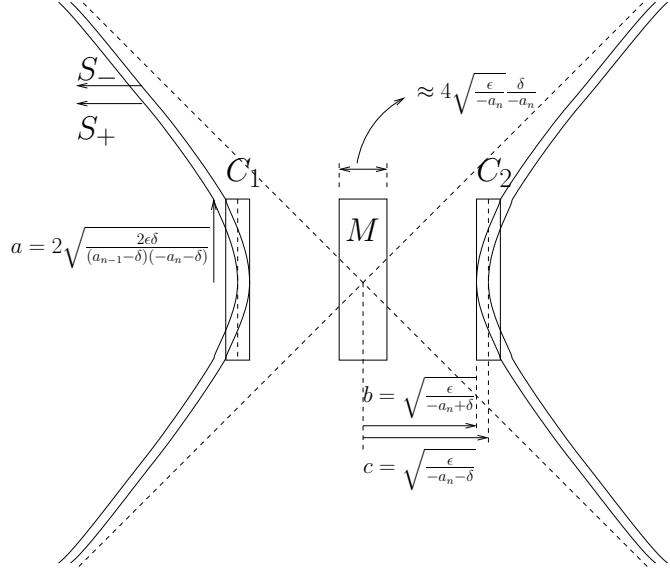}\caption{\label{fig:Local-saddle}Local structure of saddle point.}

\end{figure}

\textbf{Step 1: Calculate variables in Figure \ref{fig:Local-saddle}.}

The two points in distinct components of $S_{+}$ closest to each
other are the points $\left(\mathbf{0},\pm\sqrt{\frac{\epsilon}{-a_{n}-\delta}}\right)$,
and one easily calculates the values of $b$ and $c$ (which are the
distances between $\mathbf{0}$ and $S_{-}$, and that of $\mathbf{0}$
and $S_{+}$ respectively) in the diagram to be $\sqrt{\frac{\epsilon}{-a_{n}+\delta}}$
and $\sqrt{\frac{\epsilon}{-a_{n}-\delta}}$. Thus the distance between
the two components of $S$ is at most $2\sqrt{\frac{\epsilon}{-a_{n}-\delta}}$.
The points in $S$ that minimize the distance between the components
must lie in two cylinders $C_{1}$ and $C_{2}$ defined by \begin{eqnarray}
C_{1} & := & \mathbb{B}^{n-1}(\mathbf{0},a)\times\left[b-2c,-b\right]\subset\mathbb{R}^{n-1}\times\mathbb{R},\nonumber \\
C_{2} & := & \mathbb{B}^{n-1}(\mathbf{0},a)\times\left[b,2c-b\right]\subset\mathbb{R}^{n-1}\times\mathbb{R},\label{eq:The-cylinders}\end{eqnarray}
for some $a>0$. In other words, $C_{1}$ and $C_{2}$ are cylinders
with spherical base of radius $a$ such that \[
(S_{-}\backslash S_{+})\cap\left(\mathbb{R}^{n-1}\times\left[b-2c,2c-b\right]\right)\cap\mathbb{B}(\mathbf{0},\theta)\subset C_{1}\cup C_{2}.\]
They are represented as the left and right rectangles in Figure \ref{fig:Local-saddle}. 

We now find a value of $a$. We can let $x(n)=2c-b$, and we need\begin{eqnarray*}
\sum_{j=1}^{n-1}(a_{j}-\delta)x^{2}(j)+(a_{n}-\delta)x^{2}(n) & \leq & -\epsilon\\
\Rightarrow\sum_{j=1}^{n-1}(a_{j}-\delta)x^{2}(j)+(a_{n}-\delta)\left(2\sqrt{\frac{\epsilon}{-a_{n}-\delta}}-\sqrt{\frac{\epsilon}{-a_{n}+\delta}}\right)^{2} & \leq & -\epsilon.\end{eqnarray*}
Continuing the arithmetic gives

\begin{eqnarray*}
 &  & \sum_{j=1}^{n-1}(a_{j}-\delta)x^{2}(j)\\
 & \leq & \epsilon\left(-1-(a_{n}-\delta)\left(\frac{4}{-a_{n}-\delta}+\frac{1}{-a_{n}+\delta}-\frac{4}{\sqrt{-a_{n}-\delta}\sqrt{-a_{n}+\delta}}\right)\right)\\
 & \leq & \epsilon\left(-1-(a_{n}-\delta)\left(\frac{4}{-a_{n}-\delta}+\frac{1}{-a_{n}+\delta}-\frac{4}{-a_{n}+\delta}\right)\right)\\
 & = & \frac{8\epsilon\delta}{-a_{n}-\delta}.\end{eqnarray*}
The radius is maximized when $x(1)=x(2)=\cdots=x(n-2)=0$ and $x(n-1)=2\sqrt{\frac{2\epsilon\delta}{(a_{n-1}-\delta)(-a_{n}-\delta)}}$,
which gives our value of $a$.

\textbf{Step 2: }$(\lev_{\leq-\epsilon}f)\cap\mathring{\mathbb{B}}(\mathbf{0},\theta)$\textbf{
has exactly two components if $\epsilon$ is small enough.}

Note that $(\lev_{\leq-\epsilon}f)\cap\mathbb{B}(\mathbf{0},\theta)$
does not intersect the subspace $L^{\prime}:=\left\{ x\mid x(n)=0\right\} $,
since $f(x)\geq0$ for all $x\in L^{\prime}\cap\mathbb{B}(\mathbf{0},\theta)$.
We proceed to show that \[
U_{<}:=\left\{ x\mid x(n)<0\right\} \cap\mathring{\mathbb{B}}(\mathbf{0},\theta)\]
 contains exactly one path connected component if $\epsilon$ is small
enough. A similar statement for $U_{>}$ defined in a similar way
will allow us to conclude that $(\lev_{\leq-\epsilon}f)\cap\mathring{\mathbb{B}}(\mathbf{0},\theta)$
has exactly two components.

Consider two points $v_{1},v_{2}$ in $(\lev_{\leq-\epsilon}f)\cap U_{<}$.
We want to find a path connecting $v_{1}$ and $v_{2}$ and contained
in $(\lev_{\leq-\epsilon}f)\cap U_{<}$. We may assume that $v_{1}(n)\leq v_{2}(n)<0$.
By the continuity of the Hessian, assume that $\theta$ is small enough
so that for all $x\in\mathbb{B}(\mathbf{0},\theta)$, the top left
principal submatrix of $H(x)$ corresponding to the first $n-1$ elements
is positive definite. Consider the subspace $L^{\prime}(\alpha):=\{x\mid x(n)=\alpha\}$.
The positive definiteness of the submatrix of $H(x)$ on $\mathbb{B}(\mathbf{0},\theta)$
tells us that $f$ is strictly convex on $\mathbb{B}(\mathbf{0},\theta)\cap L^{\prime}(\alpha)$.

If $v_{1}(n)=v_{2}(n)$, then the line segment connecting $v_{1}$
and $v_{2}$ lies in $(\lev_{\leq-\epsilon}f)\cap L^{\prime}(v_{1}(n))\cap\mathring{\mathbb{B}}(\mathbf{0},\theta)$
by the convexity of $f$ on $L^{\prime}(v_{1}(n))\cap\mathring{\mathbb{B}}(\mathbf{0},\theta)$.
Otherwise, assume that $v_{1}(n)<v_{2}(n)$. 

Here is a lemma that we will need for the proof.
\begin{lem}
\label{lem:Can-move-left}Suppose Assumption \ref{ass:simplify-quad}
holds. We can reduce $\theta>0$ and $\delta>0$ if necessary so that
Assumption \ref{ass:simplify-theta} is satisfied, and the $n$th
component of $\nabla f(x)$ is positive for all $x\in(\lev_{\leq0}f)\cap\mathbb{B}(\mathbf{0},\theta)\cap\{x\mid x(n)<0\}$.\end{lem}
\begin{proof}
We first define $\tilde{S}_{-}$ by \[
\tilde{S}_{-}:=\{x\mid(a_{n-1}-\delta)\sum_{j=1}^{n-1}x^{2}(j)+(a_{n}-\delta)x^{2}(n)\leq0\}.\]
It is clear that $(a_{n-1}-\delta)\sum_{j=1}^{n-1}x^{2}(j)+(a_{n}-\delta)x^{2}(n)\leq f(x)$
for all $x\in\mathbb{B}(\mathbf{0},\theta)$, so $(\lev_{\leq0}f)\cap\mathbb{B}(\mathbf{0},\theta)\subset\tilde{S}_{-}\cap\mathbb{B}(\mathbf{0},\theta)$.

We now use the expansion $\nabla f(x)=H(\mathbf{0})x+o(\left|x\right|)$,
and prove that the $n$th component of $\nabla f(x)$ is negative
for all $x\in\tilde{S}_{-}\cap\mathbb{B}(\mathbf{0},\theta)\cap\{x\mid x(n)<0\}$.
We can reduce $\theta$ so that $\left|\nabla f(x)-H(\mathbf{0})x\right|<\delta\left|x\right|$
for all $x\in\mathbb{B}(\mathbf{0},\theta)$. Note that if $x\in\tilde{S}_{-}$,
then \begin{eqnarray*}
(a_{n-1}-\delta)\sum_{j=1}^{n-1}x^{2}(j)+(a_{n}-\delta)x^{2}(n) & \leq & 0\\
\Rightarrow(a_{n-1}-\delta)\left|x\right|^{2}+(a_{n}-a_{n-1})x^{2}(n) & \leq & 0\\
\Rightarrow\left|x\right| & \leq & \sqrt{\frac{a_{n-1}-a_{n}}{a_{n-1}-\delta}}\left(-x(n)\right).\end{eqnarray*}
The $n$th component of $\nabla f(x)$ is bounded from below by\[
a_{n}x(n)-\delta\left|x\right|\leq a_{n}x(n)+\delta\sqrt{\frac{a_{n-1}-a_{n}}{a_{n-1}-\delta}}x(n).\]
Provided that $\delta$ is small enough, the term above is positive
since $x(n)<0$.
\end{proof}
We now return to show that there is a path connecting $v_{1}$ and
$v_{2}$. Note that $S_{+}\cap\mathring{\mathbb{B}}(\mathbf{0},\theta)\cap\{x\mid x(n)<0\}$
is a convex set. (To see this, note that $S_{+}\cap\{x\mid x(n)<0\}$
can be rotated so that it is the epigraph of a convex function.) Since
$S_{+}\cap\mathring{\mathbb{B}}(\mathbf{0},\theta)\subset(\lev_{\leq-\epsilon}f)\cap\mathring{\mathbb{B}}(\mathbf{0},\theta)$,
the open line segment connecting the points $(\mathbf{0},-\theta),(\mathbf{0},-c)\in\mathbb{R}^{n-1}\times\mathbb{R}$
lies in $(\lev_{\leq-\epsilon}f)\cap\mathring{\mathbb{B}}(\mathbf{0},\theta)$.
If $-\theta<v_{1}(n)<v_{2}(n)\leq-c$, the piecewise linear path connecting
$v_{2}$ to $(\mathbf{0},v_{2}(n))$ to $(\mathbf{0},v_{1}(n))$ to
$v_{1}$ lies in $(\lev_{\leq-\epsilon}f)\cap\mathring{\mathbb{B}}(\mathbf{0},\theta)$.

In the case when $v_{2}(n)>-c$, we see that $v_{2}$ must lie in
$C_{1}$. Lemma \ref{lem:Can-move-left} tells us that the line segment
joining $v_{2}$ and $v_{2}+(\mathbf{0},-c-v_{2}(n))$ lies in $(\lev_{\leq-\epsilon}f)\cap\mathring{\mathbb{B}}(\mathbf{0},\theta)$.
This allows us to find a path connecting $v_{2}$ to $v_{1}$.

\textbf{Step 3: $\tilde{x}$ and $\tilde{y}$ lie in }$\mathring{\mathbb{B}}(\mathbf{0},\theta)$\textbf{.}

The points $\tilde{x}$ and $\tilde{y}$ must lie in $C_{1}$ and
$C_{2}$ respectively, and both $C_{1}$ and $C_{2}$ lie in $\mathring{\mathbb{B}}(\mathbf{0},\theta)$
if $\epsilon$ is small enough. Therefore, we can minimize over the
compact sets $(\lev_{\leq-\epsilon}f)\cap C_{1}$ and $(\lev_{\leq-\epsilon}f)\cap C_{2}$,
which tells us that a minimizing pair $(\tilde{x},\tilde{y})$ exist.
\end{proof}
In fact, under the assumptions of Proposition \ref{pro:alg-prelim},
$\tilde{x}$ and $\tilde{y}$ are unique, but all we need in the proof
of Proposition \ref{pro:alpha-bound} below is that $\tilde{x}$ and
$\tilde{y}$ lie in the sets $C_{1}$ and $C_{2}$ defined by \eqref{eq:The-cylinders}
respectively ans represented as rectangles in Figure \ref{fig:Local-saddle}.
We defer the proof of uniqueness to Proposition \ref{pro:2-convex-sets}.

Our next result is on a bound for possible locations of $z_{i}$ in
step 1(b). 
\begin{prop}
\label{pro:alpha-bound}Suppose that $f:\mathbb{R}^{n}\rightarrow\mathbb{R}$
is $\mathcal{C}^{2}$, and $\bar{x}$ is a nondegenerate critical
point of Morse index 1 such that $f(\bar{x})=c$. If $\theta$ is
small enough, then for all small $\epsilon>0$ (depending on $\theta$),
\begin{enumerate}
\item [(1)] Two closest points of the two components of $(\lev_{\leq c-\epsilon}f)\cap\mathring{\mathbb{B}}(\bar{x},\theta)$,
say $\tilde{x}$ and $\tilde{y}$, exist, 
\item [(2)] For any such points $\tilde{x}$ and $\tilde{y}$,  $f$ is
strictly convex on $L\cap\mathring{\mathbb{B}}(\bar{x},\theta)$,
where $L$ is the orthogonal bisector of $\tilde{x}$ and $\tilde{y}$,
and 
\item [(3)] $f$ has a unique minimizer on $L\cap\mathring{\mathbb{B}}(\bar{x},\theta)$.
Furthermore, $\min_{L\cap\mathring{\mathbb{B}}(\mathbf{0},\theta)}f\leq f(\bar{x})\leq\max_{[\tilde{x},\tilde{y}]}f$.
\end{enumerate}
\end{prop}
\begin{proof}
Suppose that Assumption \ref{ass:simplify-quad} holds, and choose
$\delta\in(0,\min\{a_{n-1},-a_{n}\})$. Suppose that $\theta>0$ is
small enough such that Assumption \ref{ass:simplify-theta} holds.
Throughout this proof, we assume all vectors accented with a hat  '$\wedge$'
are of Euclidean length 1. It is clear that $f(\tilde{x})=f(\tilde{y})=-\epsilon$.
Point (1) of the result comes from Proposition \ref{pro:alg-prelim}.
We first prove the following lemma. 
\begin{lem}
\label{lem:alpha-step-1}Suppose Assumptions \ref{ass:simplify-quad}
and \ref{ass:simplify-theta} hold. If $\theta>0$ is small enough,
then for all small $\epsilon>0$ (depending on $\theta$), two closest
points of the two components of $(\lev_{\leq-\epsilon}f)\cap\mathring{\mathbb{B}}(\mathbf{0},\theta)$,
say $\tilde{x}$ and $\tilde{y}$, exist. Let $L$ be the perpendicular
bisector of $\tilde{x}$ and $\tilde{y}$. Then  \begin{eqnarray*}
 &  & (\lev_{\le0}f)\cap L\cap\mathring{\mathbb{B}}(\mathbf{0},\theta)\subset\mathbb{B}^{n-1}\left(\mathbf{0},\alpha\sqrt{\frac{\left(-a_{n}+\delta\right)}{\left(a_{n-1}-\delta\right)}}\right)\times\left(-\alpha,\alpha\right),\\
 &  & \mbox{ where }\alpha=\delta\sqrt{\frac{\epsilon}{-a_{n}}}\left(\frac{8}{a_{n-1}}+\frac{2}{-a_{n}}\right)+o(\delta).\end{eqnarray*}
\end{lem}
\begin{proof}
By Proposition \ref{pro:alg-prelim}, the points $\tilde{x}$ and
$\tilde{y}$ must exist. We proceed to prove the rest of Lemma \ref{lem:alpha-step-1}.

\textbf{Step 1: Calculate remaining values in Figure \ref{fig:Local-saddle}.}

We calculated the values of $a$, $b$ and $c$ in step 2 of the proof
of Proposition \ref{pro:alg-prelim}, and we proceed to calculate
the rest of the variables in Figure \ref{fig:Local-saddle}. The middle
rectangle in Figure \ref{fig:Local-saddle} represents the possible
locations of midpoints of points in $C_{1}$ and $C_{2}$, and is
a cylinder as well. We call this set $M$. The radius of this cylinder
is the same as that of $C_{1}$ and $C_{2}$, and the width of this
cylinder is $4(c-b)$, which gives an $o(\delta)$ approximation\begin{eqnarray*}
4(c-b) & = & 4\left(\sqrt{\frac{\epsilon}{-a_{n}-\delta}}-\sqrt{\frac{\epsilon}{-a_{n}+\delta}}\right)\\
 & = & 4\sqrt{\frac{-a_{n}\epsilon}{(-a_{n}-\delta)(-a_{n}+\delta)}}\left(\sqrt{1+\frac{\delta}{-a_{n}}}-\sqrt{1-\frac{\delta}{-a_{n}}}\right)\\
 & = & 4\sqrt{\frac{\epsilon}{-a_{n}}}\left(\left(1+\frac{\delta}{-2a_{n}}\right)-\left(1-\frac{\delta}{-2a_{n}}\right)\right)+o(\delta)\\
 & = & 4\sqrt{\frac{\epsilon}{-a_{n}}}\frac{\delta}{-a_{n}}+o(\delta).\end{eqnarray*}
These calculations suffice for the calculations in step 2 of this
proof.

\textbf{Step 2: Set up optimization problem for bound on $(\lev_{\leq0}f)\cap L\cap\mathring{\mathbb{B}}(\mathbf{0},\theta)$.}

From the values of $a$ and $b$ calculated previously, we deduce
that a vector $c_{2}-c_{1}$, with $c_{i}\in C_{i}$, can be scaled
so that it is of the form $(\gamma\frac{a}{b}\hat{\mathbf{v}}_{1},1)$,
where $\hat{\mathbf{v}}_{1}\in\mathbb{R}^{n-1}$ is of norm $1$ and
$0\leq\gamma\leq1$. (i.e., the norm corresponding to the first $n-1$
coordinates is at most $\frac{a}{b}$.) These are possible normals
for $L$, the perpendicular bisector of $\tilde{x}$ and $\tilde{y}$.
The formula for $\frac{a}{b}$ is\begin{eqnarray*}
\frac{a}{b} & = & 2\sqrt{\frac{2\epsilon\delta}{(a_{n-1}-\delta)(-a_{n}-\delta)}}\div\sqrt{\frac{\epsilon}{-a_{n}+\delta}}\\
 & = & 2\sqrt{\frac{2\delta(-a_{n}+\delta)}{(a_{n-1}-\delta)(-a_{n}-\delta)}}.\end{eqnarray*}
So we can represent a normal of the affine space $L$ as \begin{equation}
\left(2\gamma_{1}\sqrt{\frac{2\delta(-a_{n}+\delta)}{(a_{n-1}-\delta)(-a_{n}-\delta)}}\hat{\mathbf{v}}_{1},1\right)\mbox{ for some }0\leq\gamma_{1}\leq1.\label{eq:normal-axis}\end{equation}
We now proceed to bound the minimum of $f$ on all possible perpendicular
bisectors of $c_{1}$ and $c_{2}$ within $\mathring{\mathbb{B}}(\mathbf{0},\theta)$,
where $c_{1}\in C_{1}$ and $c_{2}\in C_{2}$.  We find the largest
value of $\alpha$ such that 
\begin{itemize}
\item there is a point of the form $(\mathbf{v}_{2},\alpha)$ lying in $\tilde{S}_{-}$,
where \[
\tilde{S}_{-}:=\{x\mid(a_{n-1}-\delta)\sum_{j=1}^{n-1}x^{2}(j)+(a_{n}-\delta)x^{2}(n)\leq0\}\subset\mathbb{R}^{n-1}\times\mathbb{R}.\]

\item $\left(\mathbf{v}_{2},\alpha\right)\in\tilde{L}$ for some affine
space $\tilde{L}$ passing through a point $p\in M$ and having a
normal vector of the form in Formula \eqref{eq:normal-axis}.
\end{itemize}
The set $\tilde{S}_{-}$ is the same as that defined in the proof
of Lemma \ref{lem:Can-move-left}. Note that $\tilde{S}_{-}\cap\mathring{\mathbb{B}}(\mathbf{0},\theta)\supset(\lev_{\leq0}f)\cap\mathring{\mathbb{B}}(\mathbf{0},\theta)$,
and this largest value of $\alpha$ is an upper bound on the absolute
value of the $n$th coordinate of elements in $(\lev_{\leq0}f)\cap L\cap\mathring{\mathbb{B}}(\mathbf{0},\theta)$.

\textbf{Step 3: Solving for $\alpha$.}

For a point $(\mathbf{v}_{2},\alpha)\in\tilde{S}_{-}$, where $\mathbf{v}_{2}=(x(1),x\left(2\right),\dots,x(n-1))\in\mathbb{R}^{n-1}$,
we have\begin{eqnarray*}
(a_{n-1}-\delta)\sum_{j=1}^{n-1}x^{2}(j)+(a_{n}-\delta)\alpha^{2} & \leq & 0.\\
\Rightarrow\left|\mathbf{v}_{2}\right|^{2} & = & \sum_{j=1}^{n-1}x^{2}(j)\\
 & \leq & \frac{(-a_{n}+\delta)}{(a_{n-1}-\delta)}\alpha^{2}.\\
\Rightarrow\left|\mathbf{v}_{2}\right| & \leq & \sqrt{\frac{(-a_{n}+\delta)}{(a_{n-1}-\delta)}}\alpha.\end{eqnarray*}
Therefore, we can write $\left(\mathbf{v}_{2},\alpha\right)$ as\begin{equation}
\left(\gamma_{2}\sqrt{\frac{(-a_{n}+\delta)}{(a_{n-1}-\delta)}}\alpha\hat{\mathbf{v}}_{2},\alpha\right),\label{eq:max-alpha}\end{equation}
where $\hat{\mathbf{v}}_{2}\in\mathbb{R}^{n-1}$ is a vector of unit
norm, and $0\leq\gamma_{2}\le1$.  We can assume that $p$ has coordinates\[
\left(2\gamma_{3}\sqrt{\frac{2\epsilon\delta}{(a_{n-1}-\delta)(-a_{n}-\delta)}}\hat{\mathbf{v}}_{3},2\gamma_{4}\sqrt{\frac{\epsilon}{-a_{n}}}\frac{\delta}{-a_{n}}+o(\delta)\right),\]
where $\hat{\mathbf{v}}_{3}\in\mathbb{R}^{n-1}$ is some vector of
unit norm, and $0\leq\gamma_{3},\gamma_{4}\leq1$. Note that the $n$th
component is half the width of $M$. Hence a possible tangent on $\tilde{L}$
is \[
\left(\gamma_{1}\sqrt{\frac{(-a_{n}+\delta)}{(a_{n-1}-\delta)}}\alpha\hat{\mathbf{v}}_{2},\alpha\right)-\left(2\gamma_{3}\sqrt{\frac{2\epsilon\delta}{(a_{n-1}-\delta)(-a_{n}-\delta)}}\hat{\mathbf{v}}_{3},2\gamma_{4}\sqrt{\frac{\epsilon}{-a_{n}}}\frac{\delta}{-a_{n}}+o(\delta)\right).\]
To simplify notation, note that we only require an $O(\delta)$ approximation
of $\alpha$, we can take the terms like $-a_{n}+\delta$ and $-a_{n}-\delta$
to be $-a_{n}+O(\delta)$ and so on. The dot product of the above
vector and the normal of the affine space $L$ calculated in Formula
\eqref{eq:normal-axis} must be zero, which after some simplification
gives:\begin{eqnarray*}
\Bigg(\left(\gamma_{2}\sqrt{\frac{-a_{n}}{a_{n-1}}}+O(\delta)\right)\alpha\hat{\mathbf{v}}_{2}-\left(2\gamma_{3}\sqrt{\frac{2\epsilon\delta}{a_{n-1}(-a_{n})}}+O(\delta^{3/2})\right)\hat{\mathbf{v}}_{3}\qquad\qquad\\
,\alpha-\left(2\gamma_{4}\sqrt{\frac{\epsilon}{-a_{n}}}\frac{\delta}{-a_{n}}+o(\delta)\right)\Bigg)\cdot\left(\left(2\gamma_{1}\sqrt{\frac{2\delta}{a_{n-1}}}+O(\delta^{3/2})\right)\hat{\mathbf{v}}_{1},1\right) & = & 0.\end{eqnarray*}
At this point, we remind the reader that the $O(\delta^{k})$ terms
mean that there exists some $K>0$ such that if $\delta$ were small
enough, we can find terms $t_{1}$ to $t_{3}$ such that $\left|t_{i}\right|<K\delta^{k}$
and the formula above is satisfied by $t_{i}$ in place of the $O(\delta^{k})$
terms. Further arithmetic gives\begin{eqnarray*}
 &  & 4\gamma_{1}\gamma_{3}\sqrt{\frac{2\delta}{a_{n-1}}}\sqrt{\frac{2\epsilon\delta}{a_{n-1}(-a_{n})}}(\hat{\mathbf{v}}_{3}\cdot\hat{\mathbf{v}}_{1})+2\gamma_{4}\sqrt{\frac{\epsilon}{-a_{n}}}\frac{\delta}{-a_{n}}+o(\delta)\\
 &  & \,\,\,\,\,\,\,\,\,\,=\alpha\left(1+2\gamma_{1}\gamma_{2}\sqrt{\frac{2\delta}{a_{n-1}}}\sqrt{\frac{-a_{n}}{a_{n-1}}}(\hat{\mathbf{v}}_{2}\cdot\hat{\mathbf{v}}_{1})+o(\delta^{3/2})\right)\\
 &  & \,\,\,\,\,\,\,\,\,\,=\alpha(1+O(\sqrt{\delta}))\end{eqnarray*}
To find an upper bound for $\alpha$, it is clear that we should take
$\gamma_{1}=\gamma_{3}=\gamma_{4}=1$ and $\hat{\mathbf{v}}_{3}\cdot\hat{\mathbf{v}}_{1}=1$.
 The $O(\sqrt{\delta})$ term is superfluous, and this simplifies
to give \begin{equation}
\alpha\leq\delta\sqrt{\frac{\epsilon}{-a_{n}}}\left(\frac{8}{a_{n-1}}+\frac{2}{-a_{n}}\right)+o(\delta).\label{eq:alpha-formula}\end{equation}
We could find the minimum possible value of $\alpha$ by these same
series of steps and show that the absolute value would be bounded
above by the same bound. This ends the proof of Lemma \ref{lem:alpha-step-1}. 
\end{proof}
It is clear that the minimum value of $f$ on $L\cap\mathring{\mathbb{B}}(\mathbf{0},\theta)$
is at most $0$, since $L$ intersects the axis corresponding to the
$n$th coordinate and $f$ is nonpositive there. Therefore the set
$(\lev_{\leq0}f)\cap L\cap\mathring{\mathbb{B}}(\mathbf{0},\theta)$
is nonempty, and $f$ has a local minimizer on $L\cap\mathring{\mathbb{B}}(\mathbf{0},\theta)$.

We now state and prove our second lemma that will conclude the proof
of Proposition \ref{pro:alpha-bound}.
\begin{lem}
\label{pro:convex-concave}Let $L$ be the perpendicular bisector
of $\tilde{x}$ and $\tilde{y}$ as defined in point (1) of Proposition
\ref{pro:alpha-bound} with $\bar{x}=\mathbf{0}$. If $\delta$ and
$\theta$ are small enough satisfying Assumptions \ref{ass:simplify-quad}
and \ref{ass:simplify-theta}, then $f\mid_{L\cap\mathbb{B}(\mathbf{0},\theta)}$
is strictly convex. \end{lem}
\begin{proof}
The lineality space of $L$, written as $\mbox{lin}(L)$, is the
space of vectors orthogonal to $\tilde{x}-\tilde{y}$. We can infer
from Formula \eqref{eq:normal-axis} that $\tilde{x}-\tilde{y}$ is
a scalar multiple of a vector of the form $(w,1)$, where $w\in\mathbb{R}^{n-1}$
satisfies $\left|w\right|\rightarrow\mathbf{0}$ as $\delta\rightarrow0$.
We consider a vector $v\in\mbox{lin}(L)$ orthogonal to $(w,1)$ that
can be scaled so that $v_{}=(\tilde{w},1)$, where $(w_{},1)\cdot(\tilde{w},1)=0$,
which gives $w_{}\cdot\tilde{w}=-1$. The Cauchy Schwarz inequality
gives us\begin{eqnarray*}
\left|\tilde{w}\right|\left|w\right| & \geq & \left|\tilde{w}\cdot w\right|\\
 & = & 1\\
\Rightarrow\left|\tilde{w}_{}\right| & \geq & \left|w\right|^{-1}.\end{eqnarray*}
So \begin{eqnarray*}
\frac{v^{\top}H(p)v}{v_{}^{\top}v} & = & \frac{v_{}^{\top}H(\mathbf{0})v}{v^{\top}v}+\frac{v_{}^{\top}(H(p)-H(\mathbf{0}))v}{v_{}^{\top}v}\\
 & = & \frac{\sum_{j=1}^{n-1}a_{j}v^{2}(j)+a_{n}}{\sum_{j=1}^{n-1}v^{2}(j)+1}+\frac{v_{}^{\top}(H(p)-H(\mathbf{0}))v_{}}{v_{}^{\top}v_{}}\\
 & \geq & \underbrace{\frac{a_{n-1}\sum_{j=1}^{n-1}v^{2}(j)+a_{n}}{\sum_{j=1}^{n-1}v^{2}(j)+1}}_{(1)}+\underbrace{\frac{v^{\top}(H(p)-H(\mathbf{0}))v}{v_{}^{\top}v}}_{\left(2\right)}.\end{eqnarray*}
Since $\sum_{j=1}^{n-1}v_{}^{2}(j)=\left|\tilde{w}_{}\right|^{2}\rightarrow\infty$
as $\left|w\right|\rightarrow0$, the limit of term $(1)$ is $a_{n-1}$,
so there is an open set $\mathbb{B}(\mathbf{0},\theta)$ containing
$\mathbf{0}$ such that $\frac{v^{\top}H(p)v}{v_{}^{\top}v}>\frac{1}{2}a_{n-1}$
for all $v\in\mbox{lin}(L)\cap\{x\mid x(n)=1\}$ and $p\in\mathbb{B}(\mathbf{0},\theta)$.
By the continuity of the Hessian, we may reduce $\theta$ if necessary
so that $\left\Vert H(p)-H(\mathbf{0})\right\Vert <\frac{1}{2}a_{n-1}$
for all $p\in\mathbb{B}(\mathbf{0},\theta)$. Thus $\frac{v_{}^{\top}H(p)v}{v^{\top}v}>0$
for all $p\in\mathbb{B}(\mathbf{0},\theta)$ and $v\in\mbox{lin}(L)\cap\{x\mid x(n)=1\}$
if $\delta$ is small enough. 

The vectors of the form $v=(\tilde{w},0)$ do not present additional
difficulties as the corresponding term $(1)$ is at least $a_{n-1}$.
This proves that the Hessian $H(p)$ restricted to $\mbox{lin}(L)$
is positive definite, and hence the strict convexity of $f$ on $L\cap\mathring{\mathbb{B}}(\mathbf{0},\theta)$. 
\end{proof}
Since $f$ has a local minimizer in $L\cap\mathring{\mathbb{B}}(\mathbf{0},\theta)$
and is strictly convex there, we have (2) and the first part of part
(3). The inequality $f(\bar{x})\leq\max_{[\tilde{x},\tilde{y}]}f$
follows easily from the fact that the line segment $[\tilde{x},\tilde{y}]$
intersects the set $\{x\mid x(n)=0\}$, on which $f$ is nonnegative.
\end{proof}
Here is our theorem on the convergence of Algorithm \ref{alg:(Mountain-pass-1)}.
\begin{thm}
\label{thm:superline-conv}Suppose that $f:\mathbb{R}^{n}\rightarrow\mathbb{R}$
is $\mathcal{C}^{2}$ in a neighborhood of a nondegenerate critical
point $\bar{x}$ of Morse index 1. If $\theta>0$ is sufficiently
small and $x_{0}$ and $y_{0}$ are chosen such that 
\begin{enumerate}
\item [(a)] $x_{0}$ and $y_{0}$ lie in the two different components
of $(\lev_{\leq f(x_{0})}f)\cap\mathring{\mathbb{B}}(\bar{x},\theta)$, 
\item [(b)] $f(x_{0})=f(y_{0})<f(\bar{x})$, 
\end{enumerate}
then Algorithm \ref{alg:(Mountain-pass-1)} with $U=\mathring{\mathbb{B}}(\bar{x},\theta)$
generates a sequence of  iterates $\left\{ \tilde{x}_{i}\right\} _{i}$
and $\left\{ \tilde{y}_{i}\right\} _{i}$ lying in $\mathring{\mathbb{B}}(\bar{x},\theta)$
such that the function values $\left\{ f(\tilde{x}_{i})\right\} _{i}$
and $\left\{ f(\tilde{y}_{i})\right\} _{i}$ converge to $f(\bar{x})$
superlinearly, and the iterates $\left\{ \tilde{x}_{i}\right\} _{i}$
and $\left\{ \tilde{y}_{i}\right\} _{i}$ converge to $\bar{x}$ superlinearly.\end{thm}
\begin{proof}
As usual, suppose Assumption \ref{ass:simplify-quad} holds, and $\delta$
and $\theta$ are chosen so that Assumption \ref{ass:simplify-theta}
holds. 

\textbf{Step 1: Linear convergence of $f(\tilde{x}_{i})$ to critical
value $f(\bar{x})$.} 

Let $\epsilon=f(\tilde{x}_{i})$. The next iterate $x_{i+1}$ satisfies
$f(x_{i+1})=f(z_{i})$, and is bounded from below by \[
f(x_{i+1})\geq(a_{n}-\delta)\alpha^{2}=-\epsilon\delta^{2}\left(\frac{8}{a_{n-1}}+\frac{2}{-a_{n}}\right)^{2}+o(\delta^{2}),\]
where $\alpha$ is the value calculated in Lemma \ref{lem:alpha-step-1}.
The ratio between the previous function value and the next function
value is at most \[
\rho(\delta):=\delta^{2}\left(\frac{8}{a_{n-1}}+\frac{2}{-a_{n}}\right)^{2}+o(\delta^{2}).\]
This ratio goes to $0$ as $\delta\searrow0$, so we can choose some
$\delta$ small enough so that $\rho<\frac{1}{2}$. We can choose
$\theta$ corresponding to the value of $\delta$ satisfying Assumption
\ref{ass:simplify-theta}. This shows that the convergence to $0$
of the function values $f(\tilde{x}_{i+1})=f(x_{i+1})$ in Algorithm
\ref{alg:(Mountain-pass-1)} is linear provided $x_{0}$ and $y_{0}$
lie in $\mathbb{B}(\mathbf{0},\theta)$ and $\epsilon$ is small enough
by Proposition \ref{pro:alg-prelim}. We can reduce $\theta$ if necessary
so that $f(x)\geq-\epsilon$ for all $x\in\mathbb{B}(\mathbf{0},\theta)$,
so the condition on $\epsilon$ does not present difficulties. 

\textbf{Step 2: Superlinear convergence of $f(\tilde{x}_{i})$ to
critical value $f(\bar{x})$.}

Choose a sequence $\{\delta_{k}\}_{k}$ so that $\delta_{k}\searrow0$
monotonically. Corresponding to $\delta_{k}$, we can choose $\theta_{k}$
satisfying Assumption \ref{ass:simplify-theta}. Since $\{\tilde{x}_{i}\}_{i}$
and $\{\tilde{y}_{i}\}_{i}$ converge to $\mathbf{0}$, for any $k\in\mathbb{Z}_{+}$,
we can find some $i^{*}\in\mathbb{Z}_{+}$ so that the cylinders $C_{1}$
and $C_{2}$ constructed in Figure \ref{fig:Local-saddle} corresponding
to $\epsilon_{i}=-f(\tilde{x}_{i})$ and $\delta=\delta_{1}$ lie
wholly in $\mathbb{B}(\mathbf{0},\theta_{k})$ for all $i>i^{*}$.
As remarked in step 3 of the proof of Proposition \ref{pro:alg-prelim},
$\tilde{x}_{i}$ and $\tilde{y}_{i}$ must lie inside $C_{1}$ and
$C_{2}$, so we can take $\delta=\delta_{k}$ for the ratio $\rho$.
This means that $\frac{\left|f(\tilde{x}_{i+1})\right|}{\left|f(\tilde{x}_{i})\right|}\leq\rho(\delta_{k})$\textbf{
}for all $i>i^{*}$. As $\rho(\delta)\searrow0$ as $\delta\searrow0$,
this means that we have superlinear convergence of the $f(\tilde{x}_{i})$
to the critical value $f(\bar{x})$. 

\textbf{Step 3: Superlinear convergence of $\tilde{x}_{i}$ to the
critical point $\bar{x}$.}

We now proceed to prove that the distance between the critical point
$\mathbf{0}$ and the iterates decrease superlinearly by calculating
the value $\frac{\left|\tilde{x}_{i+1}\right|}{\left|\tilde{x}_{i}\right|}$,
or alternatively $\frac{\left|\tilde{x}_{i+1}\right|^{2}}{\left|\tilde{x}_{i}\right|^{2}}$.
The value $\left|\tilde{x}_{i}\right|$ satisfies $\left|\tilde{x}_{i}\right|^{2}\geq b^{2}=\frac{\epsilon}{-a_{n}+\delta}$.
To find an upper bound for $\left|\tilde{x}_{i+1}\right|^{2}$, it
is instructive to look at an upper bound for $\left|\tilde{x}_{i}\right|^{2}$
first. As can be deduced from Figure \ref{fig:Local-saddle}, an upper
bound for $\left|\tilde{x}_{i}\right|^{2}$ is the square of the distance
between $\mathbf{0}$ and the furthest point in $C_{1}$, which is\begin{eqnarray*}
(2c-b)^{2}+a^{2} & = & (c+(c-b))^{2}+a^{2}\\
 & = & \frac{\epsilon}{-a_{n}-\delta}+8\frac{\epsilon\delta}{(-a_{n})^{2}}+\frac{8\epsilon\delta}{(a_{n-1}-\delta)(-a_{n}-\delta)}+o(\delta).\end{eqnarray*}
This means that an upper bound for $\left|\tilde{x}_{i+1}\right|^{2}$
is \[
\delta^{2}\left(\frac{8}{a_{n-1}}+\frac{2}{-a_{n}}\right)^{2}\left(\frac{\epsilon}{-a_{n}-\delta}+\frac{8\epsilon\delta}{-a_{n}}\left(\frac{1}{-a_{n}}+\frac{1}{(a_{n-1}-\delta)}\right)\right)+o(\delta^{2}).\]
From this point, one easily sees that as $i\rightarrow\infty$, $\delta\rightarrow0$,
and $\frac{\left|\tilde{x}_{i+1}\right|^{2}}{\left|\tilde{x}_{i}\right|^{2}}\rightarrow0$.
This gives the superlinear convergence of the distance between the
critical point and the iterates $\tilde{x}_{i}$ that we seek. 
\end{proof}

\section{\label{sec:fast-local-observations}Further properties of the local
algorithm}

In this section, we take note of some interesting properties of Algorithm
\ref{alg:(Mountain-pass-1)}. First, we show that it is easy to find
$x_{i+1}$ and $y_{i+1}$ in step 2 of Algorithm \ref{alg:(Mountain-pass-1)}.
\begin{prop}
\label{pro:x-i+1-equals-z-i}Suppose the conditions in Theorem \ref{thm:superline-conv}
hold. Consider the sequence of iterates $\{x_{i}\}_{i}$ and $\{y_{i}\}_{i}$
generated by Algorithm \ref{alg:(Mountain-pass-1)}. If $i$ is large
enough, then either $x_{i+1}=z_{i}$ or $y_{i+1}=z_{i}$ in step 2
of Algorithm \ref{alg:(Mountain-pass-1)}.\end{prop}
\begin{proof}
\textit{\emph{Let $\tilde{p}:\left[0,1\right]\rightarrow\mathbb{R}^{n}$
denote the piecewise linear path connecting $x_{i}$ to $z_{i}$ to
$y_{i}$. It suffices to prove that along}} $\tilde{p}$, the function
$f$ increases to a maximum, and then decreases. Suppose Assumptions
\ref{ass:simplify-quad} and \ref{ass:simplify-theta} hold. The cylinders
$C_{1}$ and $C_{2}$ in Figure \ref{fig:Local-saddle} are loci for
$x_{i}$ and $y_{i}$. We assume that $x_{i}$ lies in $C_{2}$ in
Figure \ref{fig:Local-saddle}. The calculations in \eqref{eq:max-alpha}
in Lemma \ref{lem:alpha-step-1} tell us that $z_{i}$ can be written
as \[
\left(\sqrt{\frac{(-a_{n}+\delta)}{(a_{n-1}-\delta)}}\alpha\lambda_{1}\hat{\mathbf{v}}_{2},\lambda_{2}\alpha\right)\in\mathbb{R}^{n-1}\times\mathbb{R},\]
where $0<\lambda_{1}<\lambda_{2}\leq1$, $\left|\hat{\mathbf{v}}_{2}\right|=1$
and $\alpha=\delta\sqrt{\frac{\epsilon}{-a_{n}}}\left(\frac{8}{a_{n-1}}+\frac{2}{-a_{n}}\right)+o(\delta)$
by \eqref{eq:alpha-formula}. Therefore, $x_{i}-z_{i}$ can be written
as \[
\left(\mathbf{v}_{1},\sqrt{\frac{\epsilon}{-a_{n}+\delta}}+o(\sqrt{\delta\epsilon})\right),\]
where $\mathbf{v}_{1}\in\mathbb{R}^{n-1}$ satisfies \begin{eqnarray*}
\left|\mathbf{v}_{1}\right| & \leq & \sqrt{\frac{(-a_{n}+\delta)}{(a_{n-1}-\delta)}}\alpha+a\\
 & = & O(\sqrt{\epsilon\delta}),\end{eqnarray*}
and $a=\sqrt{\frac{2\epsilon\delta}{(a_{n-1}-\delta)(-a_{n}-\delta)}}$
is as calculated in the proof of Proposition \ref{pro:alg-prelim}.
This means that the unit vector with direction $x_{i}-z_{i}$ converges
to the $n$-th elementary vector as $\delta\searrow0$. By appealing
to Hessians as is done in the proof of Lemma \ref{pro:convex-concave},
we see that the function $f$ is strictly concave in the line segment
$[x_{i},z_{i}]$ if $i$ is large enough. Similarly, $f$ is strictly
concave in the line segment $[y_{i},z_{i}]$ if $i$ is large enough.

Next, we prove that the function $f$ has only one local maximizer
in $\tilde{p}(\left[0,1\right])$. In the case where $\nabla f(z_{i})=\mathbf{0}$,
the concavity of $f$ on the line segments $[x_{i},z_{i}]$ and $[y_{i},z_{i}]$
tells us that $z_{i}$ is the a unique maximizer on $\tilde{p}([0,1])$.
We now look at the case where $\nabla f(z_{i})\neq\mathbf{0}$. Since
$z_{i}$ is the minimizer on a subspace with normal $x_{i}-y_{i}$,
$\nabla f(z_{i})$ is a (possibly negative) multiple of $x_{i}-y_{i}$.
This means that $\nabla f(z_{i})\cdot(x_{i}-z_{i})$ has a different
sign than $\nabla f(z_{i})\cdot(y_{i}-z_{i})$. In other words, the
map $t\mapsto f(\tilde{p}\left(t\right))$ increases then decreases.
This concludes the proof of the proposition.\end{proof}
\begin{rem}
Note that in Algorithm \ref{alg:(Mountain-pass-1)}, all we need in
step 1 is a good lower bound of the critical value. We can exploit
convexity as proved in Lemma \ref{pro:convex-concave} and use cutting
plane methods to attain a lower bound for $f$ on $L_{i}\cap\mathbb{B}(\bar{x},\theta)$.
\end{rem}
Recall from Proposition \ref{pro:alpha-bound} that $M_{i}$ is a
sequence of upper bounds of the critical value $f(\bar{x})$. While
it is not even clear that $M_{i}$ is monotonically decreasing, we
can prove the following convergence result on $M_{i}$. 
\begin{prop}
\label{pro:upper-bound-better}Suppose that $f:\mathbb{R}^{n}\rightarrow\mathbb{R}$
is $\mathcal{C}^{2}$ in a neighborhood of a nondegenerate critical
point $\bar{x}$ of Morse index 1, the neighborhood $U$ of $\bar{x}$
and the points $x_{0}$ and $y_{0}$ are chosen satisfying the conditions
in the statement of Theorem \ref{thm:superline-conv}. Then in Algorithm
\ref{alg:(Mountain-pass-1)}, $M_{i}:=\max_{[x_{i},y_{i}]}f$ converges
R-superlinearly to the critical value.\end{prop}
\begin{proof}
Suppose Assumption \ref{ass:simplify-quad} holds. An upper bound
of the critical value of the saddle point is obtained by finding the
maximum along the line segment joining two points in $C_{1}$ and
$C_{2}$, which is bounded from above by \begin{eqnarray*}
(a_{1}+\delta)a^{2} & = & (a_{1}+\delta)\frac{8\epsilon\delta}{(a_{n-1}-\delta)(-a_{n}-\delta)}.\end{eqnarray*}
A more detailed analysis by using cylinders with ellipsoidal base
instead of circular base tell us that the maximum is bounded above
by $\frac{8\epsilon\delta}{(-a_{n}-\delta)}$ instead. If $\delta>0$
is small enough, this value is much smaller than $-f(x_{i})=\epsilon$.
As $i\rightarrow\infty$, the estimates $-f(x_{i})$ converge superlinearly
to $0$ by Theorem \ref{thm:superline-conv}, giving us what we need.

\end{proof}
Step 1(a) is important in the analysis of Algorithm \ref{alg:(Mountain-pass-1)}.
As explained earlier in Section \ref{sec:locally-superlinearly-convergent},
it may be difficult to implement this step. Algorithm \ref{alg:(Mountain-pass-1)}
may run fine without ever performing step 1(a) (see the example in
Section \ref{sec:Wilkinson-2}), but it may need to be performed occasionally
in a practical implementation. The following result tells us that
under the assumptions we have made so far, this problem is locally
a strictly convex problem with a unique solution.
\begin{prop}
\label{pro:2-convex-sets}Suppose that $f:\mathbb{R}^{n}\rightarrow\mathbb{R}$
is $\mathcal{C}^{2}$ in a neighborhood of a nondegenerate critical
point $\bar{x}$ of Morse index 1 with critical value $f(\bar{x})=c$.
Then if $\epsilon>0$ is small enough, there is a convex neighborhood
$U_{\epsilon}$ of $\bar{x}$ such that $(\lev_{\leq c-\epsilon}f)\cap U_{\epsilon}$
is a union of two disjoint convex sets.

Consequently, providing $\theta$ is sufficiently small, the pair
of nearest points guaranteed by Proposition \ref{pro:alg-prelim}(2)
are unique. \end{prop}
\begin{proof}
Suppose Assumptions \ref{ass:simplify-quad} and \ref{ass:simplify-theta}
hold.  In addition, we further assume that \[
\left|\nabla f(x)-H(x)\right|<\delta\left|x\right|\mbox{ for all }x\in\mathring{\mathbb{B}}(\mathbf{0},\theta).\]
We can choose $U_{\epsilon}$ to be the interior of $\conv(C_{1}\cup C_{2})$,
where $C_{1}$ and $C_{2}$ are the cylinders in Figure \ref{fig:Local-saddle}
and defined in the proof of Proposition \ref{pro:alg-prelim}, but
in view of Theorem \ref{thm:unique-min-pair}, we shall prove that
$U_{\epsilon}$ can be chosen to be the bigger set $\conv(\tilde{C}_{1}\cup\tilde{C}_{2})$,
where $\tilde{C}_{1}$ and $\tilde{C}_{2}$ are cylinders defined
by\begin{eqnarray*}
\tilde{C}_{1} & := & \mathbb{B}^{n-1}(\mathbf{0},\rho)\times\left[-\beta,-b\right]\subset\mathbb{R}^{n-1}\times\mathbb{R},\\
\tilde{C}_{2} & := & \mathbb{B}^{n-1}(\mathbf{0},\rho)\times\left[b,\beta\right]\subset\mathbb{R}^{n-1}\times\mathbb{R},\end{eqnarray*}
 where $\beta,\rho$ are constants to be determined. We choose $\beta$
such that \[
\mathbb{B}^{n-1}(\mathbf{0},a)\times\{\beta\}\subset\mbox{int}(S_{+}).\]
In particular, $\beta$ satisfies \begin{eqnarray*}
a^{2}(a_{1}+\delta)+\beta^{2}(a_{n}+\delta) & < & -\epsilon\\
\Rightarrow\beta^{2} & > & \frac{1}{-a_{n}-\delta}\left(\epsilon+a^{2}(a_{1}+\delta)\right)\\
 & = & \frac{\epsilon}{-a_{n}-\delta}\left(1+\frac{8\delta(a_{1}+\delta)}{(a_{n-1}-\delta)(-a_{n}-\delta)}\right)\end{eqnarray*}
We choose $\beta$ to be any value satisfying the above inequality.

Next, we choose $\rho$ to be the smallest value such that $S_{-}\cap(\mathbb{R}^{n-1}\times[-\beta,\beta])\cap\mathbb{B}(\mathbf{0},\theta)\subset\tilde{C}_{1}\cup\tilde{C}_{2}$.
This calculation is similar to the calculation of $a$, which gives
\begin{eqnarray*}
(a_{n-1}-\delta)\rho^{2}+(a_{n}-\delta)\beta^{2} & = & -\epsilon\\
\Rightarrow\rho & = & \sqrt{\frac{-\epsilon-(a_{n}-\delta)\beta^{2}}{a_{n-1}-\delta}}.\end{eqnarray*}
We shall not expand the terms, but remark that $\beta$ and $\rho$
are of $O(\sqrt{\epsilon})$.

The proof of Proposition \ref{pro:alg-prelim} tells us that $\conv(\tilde{C}_{1}\cup\tilde{C}_{2})\cap\lev_{\leq-\epsilon}f$
is a union of the two nonempty sets $\tilde{C}_{1}\cap\lev_{\leq-\epsilon}f$
and $\tilde{C}_{2}\cap\lev_{\leq-\epsilon}f$. It remains to show
that these two sets are strictly convex. 

Any point $x\in\tilde{C}_{1}$ can be written as\[
x=(\mathbf{x}^{\prime},x_{n}),\]
 where $\mathbf{x}^{\prime}\in\mathbb{R}^{n-1}$ is of norm at most
$\rho$, and $-\beta\leq x_{n}\leq-b$, where $\beta$ is as calculated
above and $b=\sqrt{\frac{\epsilon}{-a_{n}+\delta}}$ as in Figure
\ref{fig:Local-saddle}. This implies that \[
Hx=(\mathbf{x}^{\prime\prime},a_{n}x_{n}),\]
where $\mathbf{x}^{\prime\prime}$ is of norm at most $a_{1}\left|\mathbf{x}^{\prime}\right|$.
It is clear that as $\delta\downarrow0$, the unit vector in the direction
of $Hx$ converges to $(\mathbf{0},1)$. This implies that for any
$\kappa_{1}>0$, there exists some $\delta>0$ such that $\mbox{unit}(\nabla f(x))\cdot(\mathbf{0},1)\geq1-\kappa_{1}$
for all $x\in C_{1}$. (Note that $\tilde{C}_{1}$ depends on $\delta$.)
Here, $\mbox{unit}:\mathbb{R}^{n}\backslash\left\{ \mathbf{0}\right\} \rightarrow\mathbb{R}^{n}$
is the mapping of a nonzero vector to the unit vector pointing in
the same direction.

Let $z_{1}$ and $z_{2}$ be points in $\tilde{C}_{1}\cap(\lev_{\leq-\epsilon}f)$.
Suppose that $z_{1}(n)<z_{2}(n)$, and let $\mathbf{v}=(\mathbf{v}_{1},v_{2})\in\mathbb{R}^{n-1}\times\mathbb{R}$
be a unit vector in the same direction as $z_{2}-z_{1}$. We further
assume, by reducing $\theta$ and $\delta$ as necessary, that $\|H(x)-H(\mathbf{0})\|<\kappa_{2}$
for all $x\in\tilde{C}_{1}\cap(\lev_{\leq-\epsilon}f)$. Suppose $\kappa_{1}$
and $\kappa_{2}$ are small enough so that $\sqrt{2\kappa_{1}}<\sqrt{\frac{a_{n-1}-\kappa_{2}}{a_{n-1}-a_{n}}}$.

Note that $v_{2}\geq0$. Either one of these two cases on $v_{2}$
must hold. We prove that in both cases, he open line segment $(z_{1},z_{2})$
lies in the interior of $(\lev_{\leq-\epsilon}f)\cap\tilde{C}_{1}$.

\textbf{Case 1: $v_{2}>\sqrt{2\kappa_{1}}$.}

In this case, for all $x\in\tilde{C}_{1}$, we have \begin{eqnarray*}
\mathbf{v}\cdot(\mbox{unit}(\nabla f(x))) & = & \mathbf{v}\cdot(\mathbf{0},1)+\mathbf{v}\cdot(\mbox{unit}(\nabla f(x))-(\mathbf{0},1))\\
 & \geq & v_{2}-\left|\mathbf{v}\right|\left|\mbox{unit}(\nabla f(x))-(\mathbf{0},1)\right|\\
 & = & v_{2}-\left|\mbox{unit}(\nabla f(x))-(\mathbf{0},1)\right|\\
 & = & v_{2}-\sqrt{\left|\mbox{unit}(\nabla f(x))\right|^{2}+\left|(\mathbf{0},1)\right|^{2}-2\mbox{unit}(\nabla f(x))\cdot(\mathbf{0},1)}\\
 & > & v_{2}-\sqrt{2-2(1-\kappa_{1})}\\
 & = & v_{2}-\sqrt{2\kappa_{1}}\\
 & > & 0.\end{eqnarray*}
This means that along the line segment $[z_{1},z_{2}]$, the function
$f$ is strictly monotone. Therefore, if $x_{1},x_{2}\in(\lev_{\leq-\epsilon}f)\cap\tilde{C}_{1}$,
the open line segment $(z_{1},z_{2})$ lies in the interior of $(\lev_{\leq-\epsilon}f)\cap\tilde{C}_{1}$.

\textbf{Case 2: $v_{2}<\sqrt{\frac{a_{n-1}-\kappa_{2}}{a_{n-1}-a_{n}}}$.}

Let $H^{u}(\mathbf{0})$ denote the diagonal matrix of size $(n-1)\times(n-1)$
with elements $a_{1},\dots,a_{n-1}$. We have\begin{eqnarray*}
\mathbf{v}^{\top}H(x)\mathbf{v} & = & \mathbf{v}^{\top}H(\mathbf{0})\mathbf{v}+\mathbf{v}^{\top}(H(x)-H(\mathbf{0}))\mathbf{v}\\
 & > & \mathbf{v}_{1}^{\top}H^{u}(\mathbf{0})\mathbf{v}_{1}+a_{n}v_{2}^{2}-\left|\mathbf{v}\right|^{2}\|H(x)-H(\mathbf{0})\|\\
 & \geq & a_{n-1}\left|\mathbf{v}_{2}\right|^{2}+a_{n}v_{2}^{2}-\|H(x)-H(\mathbf{0})\|\\
 & > & a_{n-1}(1-v_{2}^{2})+a_{n}v_{2}^{2}-\kappa_{2}\\
 & = & a_{n-1}+v_{2}^{2}(a_{n}-a_{n-1})-\kappa_{2}\\
 & > & a_{n-1}+(\kappa_{2}-a_{n-1})-\kappa_{2}\\
 & \geq & 0\end{eqnarray*}
This means that the function $f$ is strictly convex along the line
segment $[z_{1},z_{2}]$, so if $x_{1},x_{2}\in(\lev_{\leq-\epsilon}f)\cap\tilde{C}_{1}$,
the open line segment $(z_{1},z_{2})$ lies in the interior of $(\lev_{\leq-\epsilon}f)\cap\tilde{C}_{1}$,
concluding the proof of the first part of this result. 

To prove the next statement on the uniqueness of the pair of closest
points, suppose that $(\tilde{x}^{\prime},\tilde{y}^{\prime})$ and
$(\tilde{x}^{\prime\prime},\tilde{y}^{\prime\prime})$ are distinct
pairs whose distance give the distance between the components of $(\lev_{\leq-\epsilon}f)\cap\mathbb{B}(\mathbf{0},\theta)$,
where $\mathbb{B}(\mathbf{0},\theta)$ is as stated in Proposition
\ref{pro:alg-prelim}. If $\epsilon$ is small enough, then $\conv(\tilde{C}_{1}\cup\tilde{C}_{2})$
lies in $\mathring{\mathbb{B}}(\mathbf{0},\theta)$. Then by the strict
convexity of the components of $(\lev_{\leq-\epsilon}f)\cap\conv(\tilde{C}_{1}\cup\tilde{C}_{2})$,
the pair $(\frac{1}{2}(\tilde{x}^{\prime}+\tilde{x}^{\prime\prime}),\frac{1}{2}(\tilde{y}^{\prime}+\tilde{y}^{\prime\prime}))$
lie in the same components, and the distance between this pair of
points must be the same as that for the pairs $(\tilde{x}^{\prime},\tilde{y}^{\prime})$
and $(\tilde{x}^{\prime\prime},\tilde{y}^{\prime\prime})$. The closest
points in the components of $[\frac{1}{2}(\tilde{x}^{\prime}+\tilde{x}^{\prime\prime}),\frac{1}{2}(\tilde{y}^{\prime}+\tilde{y}^{\prime\prime})]\cap\lev_{\leq-\epsilon}f$
give a smaller distance between the components of $(\lev_{\leq-\epsilon}f)\cap\mathbb{B}(\mathbf{0},\theta)$,
which contradicts the optimality of the pairs $(\tilde{x}^{\prime},\tilde{y}^{\prime})$
and $(\tilde{x}^{\prime\prime},\tilde{y}^{\prime\prime})$.
\end{proof}
Note that in the case of $\epsilon=0$, there may be no neighborhood
$U_{0}$ of $\bar{x}$ such that $U_{0}\cap(\lev_{\leq c}f)$ is a
union of two convex sets intersecting only at the critical point.
We also note that $U_{\epsilon}$ depends on $\epsilon$ in our result
above. The following example explains these restrictions.
\begin{example}
Consider the function $f:\mathbb{R}^{2}\rightarrow\mathbb{R}$ defined
by $f(x)=(x_{2}-x_{1}^{2})(x_{1}-x_{2}^{2})$. The shaded area in
Figure \ref{fig:nonconvex-level-sets} is a sketch of $\lev_{\leq0}f$. 

We now explain that the neighborhood $U_{\epsilon}$ defined in Proposition
\ref{pro:2-convex-sets} must depend on $\epsilon$ for this example.
For any open $U$ containing $\mathbf{0}$, we can always find two
points $p$ and $q$ in a component of $(\lev_{<0}f)\cap U$ such
that the line segment $[p,q]$ does not lie in $\lev_{<0}f$. This
implies that the component of $(\lev_{\leq-\epsilon}f)\cap U$ is
not convex if $0<\epsilon\leq-\max(f(p),f(q))$. $\diamond$

\begin{figure}
\includegraphics[scale=0.5]{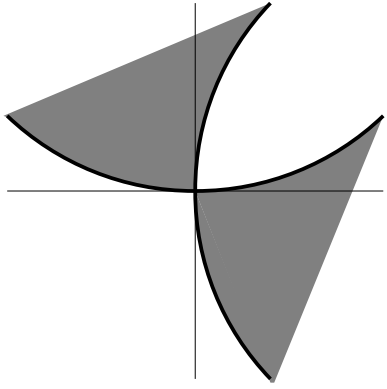}\caption{\label{fig:nonconvex-level-sets}$\lev_{\leq0}f$ for $f(x)=(x_{2}-x_{1}^{2})(x_{1}-x_{2}^{2})$}

\end{figure}

\end{example}
We now take a second look at the problem of minimizing the distance
between two components in step 1(a) of Algorithm \ref{alg:(Mountain-pass-1)}.
We need to solve the following problem for $\epsilon>0$: \begin{eqnarray}
 & \min_{x,y} & \left|x-y\right|\nonumber \\
 & \mbox{s.t.} & x\mbox{ lies in the same component as }a\mbox{ in }(\lev_{\leq f(\bar{x})-\epsilon}f)\cap\mathring{\mathbb{B}}(\bar{x},\theta)\label{eq:ideal-prob}\\
 &  & y\mbox{ lies in the same component as }b\mbox{ in }(\lev_{\leq f(\bar{x})-\epsilon}f)\cap\mathring{\mathbb{B}}(\bar{x},\theta).\nonumber \end{eqnarray}
If $(\tilde{x},\tilde{y})$ is a pair of local optimizers, then $\tilde{y}$
is the closest point to the component of $(\lev_{\leq f(\bar{x})-\epsilon}f)\cap U$
containing $\tilde{x}$ and vice versa. This gives us the following
optimality conditions: \begin{equation}
\begin{array}{r}
\nabla f(\tilde{x})=\kappa_{1}(\tilde{y}-\tilde{x}),\\
\nabla f(\tilde{y})=\kappa_{2}(\tilde{x}-\tilde{y}),\\
f(\tilde{x})=f(\bar{x})-\epsilon\\
f(\tilde{y})=f(\bar{x})-\epsilon\\
\mbox{for some }\kappa_{1},\kappa_{2}\geq0.\end{array}\label{eq:ideal-cond}\end{equation}
From Proposition \ref{pro:2-convex-sets}, we see that given any $\theta>0$
sufficiently small, provided that the conditions in Proposition \ref{pro:alg-prelim}
hold, the global minimizing pair of \eqref{eq:ideal-prob} is unique.
Even though convexity is absent, the following theorem shows that
the global minimizing pair is, under added conditions, the only pair
satisfying the optimality conditions \eqref{eq:ideal-cond}, showing
that there are no other local minimizers of \eqref{eq:ideal-prob}.
\begin{thm}
\label{thm:unique-min-pair}Suppose that $f:\mathbb{R}^{n}\rightarrow\mathbb{R}$
is $\mathcal{C}^{2}$, and $\bar{x}$ is a nondegenerate critical
point of Morse index 1 such that $f(\bar{x})=c$. If $\theta>0$ is
sufficiently small, then for any $\epsilon>0$ (depending on $\theta$)
sufficiently small, the global minimizer of \eqref{eq:ideal-prob}
is the only pair in $\mathring{\mathbb{B}}(\bar{x},\theta)\times\mathring{\mathbb{B}}(\bar{x},\theta)$
satisfying the optimality conditions \eqref{eq:ideal-cond}.\end{thm}
\begin{proof}
Suppose that Assumption \ref{ass:simplify-quad} holds, and $\delta$
is chosen small enough so that \ref{ass:simplify-theta} holds. We
also assume that $\theta$ is small enough so that $|H(x)-H(\mathbf{0})|<\frac{1}{2}\min(a_{n-1},-a_{n})$.
Seeking a contradiction, suppose that $(\tilde{x},\tilde{y})$ satisfy
the optimality conditions.

We refer to Figure \ref{fig:Local-saddle}, and also recall the definitions
of the sets $\tilde{C}_{1}$ and $\tilde{C}_{2}$ in the proof of
Proposition \ref{pro:2-convex-sets}. As proven in Proposition \ref{pro:2-convex-sets},
the convexity properties of the two level sets in $(\lev_{\leq f(\bar{x})-\epsilon}f)\cap\mathring{\mathbb{B}}(\bar{x},\theta)$
imply that if $\tilde{x}\in\tilde{C}_{1}$, $\tilde{y}\in\tilde{C}_{2}$
and the optimality conditions are satisfied, then the pair $(\tilde{x},\tilde{y})$
is the global minimizing pair.

Consider the case where $\tilde{x}\notin\tilde{C}_{1}$. Either of
the two cases hold. We note the asymmetry below in that we check whether
$\tilde{y}\in C_{2}$ instead of whether $\tilde{y}\in\tilde{C}_{2}$.

\textbf{Case 1:} $\tilde{y}\in C_{2}$: In this case, if the first
$n-1$ coordinates of $\tilde{x}$ are the same as that of $\tilde{y}$,
then $\tilde{x}$ lies in the interior of $(\lev_{\leq-\epsilon}f)\cap\mathring{\mathbb{B}}(\bar{x},\theta)$,
which is a contradiction to optimality. Recall that the value of $\beta$
was chosen such that $\tilde{y}+(\mathbf{0},\tilde{x}(n)-\tilde{y}(n))$
lies in $(\lev_{\leq-\epsilon}f)\cap\mathring{\mathbb{B}}(\bar{x},\theta)$.
By the convexity of $f|_{L^{\prime}(\tilde{x}(n))}$, where $L^{\prime}(\tilde{x}(n))$
is the affine space $\{x\mid x(n)=\tilde{x}(n)\}$, the line segment
connecting $\tilde{x}$ and $\tilde{y}+(\mathbf{0},\tilde{x}(n)-\tilde{y}(n))$
lies in $(\lev_{\leq-\epsilon}f)\cap\mathring{\mathbb{B}}(\bar{x},\theta)$.
The distance between $\tilde{y}$ and points along this line segment
decreases (at a linear rate) as one moves away from $\tilde{x}$,
which again contradicts the assumption that $(\tilde{x},\tilde{y})$
satisfy \eqref{eq:ideal-cond}.

\textbf{Case 2:} $\tilde{y}\notin C_{2}$: By the convexity of $f|_{L^{\prime}(\tilde{x}(n))}$
and $f|_{L^{\prime}(\tilde{y}(n))}$, the line segments $[\tilde{y},\tilde{y}-(\mathbf{0},\tilde{y}(n))]$
and $[\tilde{x},\tilde{x}-(\mathbf{0},\tilde{x}(n))]$ lie in $(\lev_{\leq-\epsilon}f)\cap\mathring{\mathbb{B}}(\bar{x},\theta)$.
These line segments and the optimality of the pair $(\tilde{x},\tilde{y})$
implies that the first $n-1$ components of $\tilde{x}$ and $\tilde{y}$
to be the same. This in turn implies that $\nabla f(\tilde{x})$ is
a positive multiple of $(\mathbf{0},1)$. 

Our proof ends if we show that if $\theta$ is small enough, $\nabla f(\tilde{x})$
cannot be a positive multiple of $(\mathbf{0},1)$. If $\tilde{x}\notin\tilde{C}_{1}$,
then $\tilde{x}(n)<-\beta$. If $\tilde{x}$ lies on the boundary
of $\lev_{\leq-\epsilon}f$, then $f(\tilde{x})=-\epsilon$, and we
have\begin{eqnarray*}
f(\tilde{x}) & = & -\epsilon\\
\sum_{i=1}^{n}(a_{i}+\delta)\tilde{x}(i)^{2} & \geq & -\epsilon\\
(a_{1}+\delta)\sum_{i=1}^{n}\tilde{x}(i)^{2}+(a_{n}-a_{1})\tilde{x}(n)^{2} & \geq & -\epsilon\\
(a_{1}+\delta)|\tilde{x}|^{2} & \geq & (a_{1}-a_{n})\tilde{x}(n)^{2}-\epsilon\\
\frac{|\tilde{x}|^{2}}{\tilde{x}(n)^{2}} & \geq & \frac{a_{1}-a_{n}-\frac{\epsilon}{\tilde{x}(n)^{2}}}{a_{1}+\delta}\\
 & \geq & 1+\frac{-a_{n}-\delta-\frac{\epsilon}{\beta^{2}}}{a_{1}+\delta}\end{eqnarray*}
Upon expansion of the term $\beta^{2}$ in the expression in the final
line, we see that $\frac{|\tilde{x}|^{2}}{\tilde{x}(n)^{2}}$ is bounded
from below by a constant independent of $\epsilon$ and greater than
$1$. Since $f$ is $\mathcal{C}^{2}$, the set \[
\{x\mid\nabla f(x)\mbox{ is a multiple of }(\mathbf{0},1)\}\cap\mathbb{B}(\mathbf{0},\theta)\]
is a manifold, whose tangent at the origin is the line spanned by
$(\mathbf{0},1)$. This implies that if $\theta$ is small enough,
then $\tilde{x}\notin\tilde{C}_{1}$ and $\tilde{x}$ lying on the
boundary of $\lev_{\leq-\epsilon}f$ implies that $\nabla f(\tilde{x})$
cannot be a multiple of $(\mathbf{0},1)$. We have the required contradiction.\end{proof}
\begin{rem}
\label{rem:heuristic-closest-point}We now describe a heuristic to
approximate a pair of closest points iteratively between the components
of $(\lev_{\leq c-\epsilon}f)\cap U$. For two points $x^{\prime}$
and $y^{\prime}$ that approximate $\tilde{x}_{i}$ and $\tilde{y}_{i}$,
we can find local minimizers of $f$ on the affine spaces orthogonal
to $x^{\prime}-y^{\prime}$ that pass through $x^{\prime}$ and $y^{\prime}$
respectively, say $x^{*}$, $y^{*}$, and then find the closest points
in the two components of $(\lev_{\leq c-\epsilon}f)\cap[x^{*},y^{*}]$,
where $[x^{*},y^{*}]$ is the line segment connecting $x^{*}$ and
$y^{*}$. This heuristic is particularly practical in the case of
Wilkinson problem, as we illuminate in Sections \ref{sec:Wilkinson-1}
and \ref{sec:Wilkinson-2}.
\end{rem}

\section{\label{sec:Global-convergence}Saddle points and criticality properties}

We have seen that Algorithm \ref{alg:globally-convergent-MPT} allows
us to find saddle points of mountain type. In this section, we first
prove an equivalent definition of a saddle point based on paths connecting
two points. Then we prove that saddle points are critical points in
the metric sense and in the nonsmooth sense.

In the following equivalent condition for saddle points, we say that
a path $p:[0,1]\rightarrow X$ \emph{connects} $a$ and $b$ if $p(0)=a$
and $p(1)=b$, and it is \emph{contained in $U\subset X$} if $p([0,1])\subset U$.
The \emph{maximum value} of the path $p$ is defined as $\max_{t}f\circ p(t)$.
\begin{prop}
\label{pro:equiv-mountain}Let $(X,d)$ be a metric space. For a continuous
function $f:X\rightarrow\mathbb{R}$, $\bar{x}$ is a saddle point
of mountain pass type if and only if there exists an open neighborhood
$U$ and two points $a,b\in(\lev_{<l}f)\cap U$ such that 
\begin{enumerate}
\item [(a)] The maximum value of any path connecting $a$ and $b$ contained
in $U$ is at least $f(\bar{x})$, and 
\item [(b)] for all $\epsilon>0$, there exists $\delta,\theta\in(0,\epsilon)$
and a path $p_{\epsilon}$ connecting $a$ and $b$ contained in $U$
such that the maximum value of $p_{\epsilon}$ is at most $f(\bar{x})+\epsilon$,
and $(\lev_{\geq f(\bar{x})-\theta}f)\cap p_{\epsilon}([0,1])\subset\mathbb{B}(\bar{x},\delta)$.
\end{enumerate}
\end{prop}
\begin{proof}
We first prove that the conditions (a) and (b) above imply that $\bar{x}$
is a saddle point. Let $A$ and $B$ be the path connected components
of $\lev_{<f(\bar{x})}f\cap U$ containing $a$ and $b$ respectively.
For any $\epsilon>0$, the condition $(\lev_{\geq f(\bar{x})-\theta}f)\cap p_{\epsilon}([0,1])\subset\mathbb{B}(\bar{x},\delta)$
tells us that we can find points $x_{\epsilon}\in A$ and $y_{\epsilon}\in B$
such that $d(\bar{x},x_{\epsilon})<\delta<\epsilon$ and $d(\bar{x},y_{\epsilon})<\epsilon$.
For a sequence $\epsilon_{i}\searrow0$, we set $x_{i}=x_{\epsilon_{i}}$
and $y_{i}=y_{\epsilon_{i}}$. This shows that $\bar{x}$ lies in
both the closure of $A$ and that of $B$, and hence $\bar{x}$ is
a saddle point.

Next, we prove the converse. Suppose that $\bar{x}$ is a saddle point,
with $U$ being a neighborhood of $\bar{x}$, and the sets $A$ and
$B$ are two path components of $(\lev_{<f(\bar{x})}f)\cap U$ whose
closures contain $\bar{x}$. For any $\epsilon>0$, we can find some
$\delta\in(0,\epsilon)$ such that $d(x,\bar{x})<\delta$ implies
$\left|f(x)-f(\bar{x})\right|<\epsilon$. There are two points $x_{\epsilon}\in A$
and $y_{\epsilon}\in B$ such that $d(x_{\epsilon},\bar{x})<\delta$
and $d(y_{\epsilon},\bar{x})<\delta$. 

Let $a$ and $b$ be any two points in the sets $A$ and $B$ respectively.
There is a path connecting $a$ to $x_{\epsilon}$ contained in $\lev_{<f(\bar{x})}f\cap U$,
say $p_{a}$, and we can similarly find a path $p_{b}$ connecting
$y_{\epsilon}$ to $b$ contained in $\lev_{<f(\bar{x})}f\cap U$.
The maximum values on both paths $p_{a}$ and $p_{b}$ are less than
$f(\bar{x})$, so there is some $\theta\in(0,\epsilon)$ such that
both maximum values are bounded above by $f(\bar{x})-\theta$. Choose
a path $p_{\epsilon}^{\prime}$ to be the line segment connecting
$x_{a}$ and $y_{b}$ contained in $\mathbb{B}(\bar{x},\delta)$.
The path $p_{\epsilon}$ formed by the concatenation of the paths
$p_{a}$, $p_{\epsilon}^{\prime}$ and $p_{b}$ satisfies condition
(b). Condition (a) is easily seen to be satisfied, and hence we are
done.
\end{proof}
Ideally, we want to improve condition (b) in Proposition \ref{pro:equiv-mountain}
so that $\bar{x}$ is the maximum point on some mountain pass connecting
$a$ and $b$. We shall see in Example \ref{exa:saddle-point-no-b'}
that saddle points in general need not have this property. A simple
finite dimensional condition on the function $f$ so that this happens
is semi-algebraicity. A set in $\mathbb{R}^{n}$ is  \emph{semi-algebraic
}if it is a union of finitely many sets defined by finitely many polynomial
inequalities, and a function $f:\mathbb{R}^{n}\rightarrow\mathbb{R}$
is \emph{semi-algebraic }if its graph $\{(x,y)\in\mathbb{R}^{n}\times\mathbb{R}\mid y=f(x)\}$
is a semi-algebraic set. Semi-algebraic objects remove much of the
oscillatory behavior that typically does not appear in applications,
and form a large class of objects that appear in applications. We
will appeal to semi-algebraic geometry for only the next result, and
we refer readers interested in the general theory of semi-algebraic
functions (and more generally, that of o-minimal structures and tame
topology, under which Proposition \ref{pro:semi-algebraic-mtn} also
holds) to \cite{BR90,Coste02,Coste99,vdDries98}.
\begin{prop}
\label{pro:semi-algebraic-mtn}In the case where $f:\mathbb{R}^{n}\rightarrow\mathbb{R}$
is semi-algebraic, condition (b) in Proposition \ref{pro:equiv-mountain}
can be replaced with 
\begin{enumerate}
\item [(b$^{\prime}$)] There is a path connecting $a$ and $b$ contained
in $U$ along which the unique maximizer is $\bar{x}$.
\end{enumerate}
\end{prop}
\begin{proof}
It is clear that (b$^{\prime}$) is a stronger condition than (b),
so we prove that if $f$ is semi-algebraic, then (b$^{\prime}$) holds.
Suppose $\bar{x}$ is a saddle point of mountain pass type. Let $U$
be an open neighborhood of $\bar{x}$, and sets $A$ and $B$ be two
components of $(\lev_{<f(\bar{x})}f)\cap U$ whose closures contain
$\bar{x}$. Choose points $a\in A$ and $b\in B$. It is clear that
$A$ and $B$ are semi-algebraic (see for example \cite[Section 3.2]{Coste99}.
By the curve selection lemma (see for example \cite[Section 3.1]{Coste99}),
there is a path $p_{a}$ connecting $a$ and $\bar{x}$ such that
$p_{a}(1)=\bar{x}$, and $p_{a}([0,1))\subset A$. Similarly, we can
find a path $p_{b}$ connecting $\bar{x}$ and $b$ such that $p_{b}(0)=\bar{x}$
and $p_{b}((0,1])\subset B$. The concatenation of $p_{a}$ and $p_{b}$
gives us what we need.
\end{proof}
In the absence of semi-algebraicity, the following example illustrates
that a saddle point need not satisfy condition (b$^{\prime}$). 
\begin{example}
\label{exa:saddle-point-no-b'}We define $f:\mathbb{R}^{2}\rightarrow\mathbb{R}$
through Figure \ref{fig:bad-saddle}. There are 2 shapes in the positive
quadrant the figure: a blue {}``comb'' $C$ wrapping around a brown
{}``sun'' $S$. The closure of $C$ contains the origin $\mathbf{0}$
(the intersection of the horizontal and vertical axis).

We can define a continuous $f:\mathbb{R}^{2}\rightarrow\mathbb{R}$
so that $f$ is negative on $C\cup(-C)$ and positive on $(S\cup(-S))\backslash\{\mathbf{0}\}$
and $\{(x,y)\mid xy<0\}$, and extend $f$ continuously to all of
$\mathbb{R}^{2}$ using the Tietze extension theorem. It is clear
that $\mathbf{0}$ is a saddle point, and the sets $A,B\subset\lev_{<0}f$
whose closures contain $\mathbf{0}$ can be taken to be the path connected
components containing $C$ and $(-C)$ respectively. But the origin
$\mathbf{0}$ does not satisfy condition (b$^{\prime}$).

\begin{figure}
\includegraphics[scale=0.5]{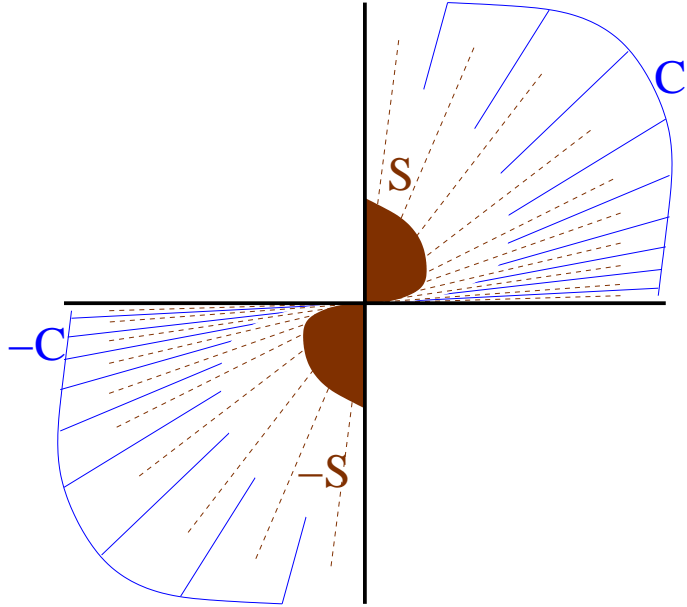}\caption{\label{fig:bad-saddle}Illustration of saddle point in Example \ref{exa:saddle-point-no-b'}.}

\end{figure}

\end{example}

Our next step is to establish the relation between saddle points and
criticality in metric spaces. We recall the following definitions
in metric critical point theory from \cite{DM94,IS96,Katriel94}.
\begin{defn}
\label{def:Deformation-critical}Let $(X,d)$ be a metric space. We
call the point $x$ \emph{Morse regular} for the function $f:X\rightarrow\mathbb{R}$
if, for some numbers $\gamma,\sigma>0$, there is a continuous function
\[
\phi:\mathbb{B}(x,\gamma)\times[0,\gamma]\rightarrow X\]
 such that all points $u\in\mathbb{B}(x,\gamma)$ and $t\in[0,\gamma]$
satisfy the inequality \[
f(\phi(x,t))\leq f(x)-\sigma t,\]
and that $\phi(\cdot,0)$ is the identity map. The point $x$ is \emph{Morse
critical }if it is not Morse regular. 

If there is some $\kappa>0$ and such a function $\phi$ that also
satisfies the inequality \[
d(\phi(x,t),x)\leq\kappa t,\]
then we call $x$ \emph{deformationally regular}. The point $x$ is
\emph{deformationally critical }if it is not deformationally regular.
\end{defn}
We now relate saddle points to Morse critical and deformationally
critical points.
\begin{prop}
For a function $f:X\rightarrow\mathbb{R}$ defined on a metric space
$X$, $\bar{x}$ is a saddle point of mountain pass type implies that
$\bar{x}$ is deformationally critical. If in addition, either $X=\mathbb{R}^{n}$
or condition (b$^{\prime}$) in Proposition \ref{pro:semi-algebraic-mtn}
holds, then $\bar{x}$ is Morse critical.\end{prop}
\begin{proof}
Let $U$ be an open neighborhood of $\bar{x}$ as defined in Definition
\ref{def:saddle-point}, and let $A$ and $B$ be two distinct components
of $(\lev_{<f(\bar{x})}f)\cap U$ which contain $\bar{x}$ in their
closures. The proofs of all three results by contradiction are similar.
For convenience, we label the following three assumptions as follows,
and prove that they all lead to the contradiction that $A$ and $B$
cannot be distinct path components in $U$.
\begin{enumerate}
\item [$(D)$] $\bar{x}$ is deformationally regular.
\item [$(M_{\mathbb{R}^{n}})$] $\bar{x}$ is Morse regular, and $X=\mathbb{R}^{n}$.
\item [$(M_{b^{\prime}})$] $\bar{x}$ is Morse regular, and condition
(b$^{\prime}$) in Proposition \ref{pro:semi-algebraic-mtn} holds.
\end{enumerate}
Suppose condition $(M_{\mathbb{R}^{n}})$ holds. Let $\gamma,\sigma>0$
and $\phi:\mathbb{B}(\bar{x},\gamma)\times[0,\gamma]\rightarrow X$
satisfy the properties of Morse regularity given in Definition \ref{def:Deformation-critical}.
We can assume that $\gamma$ is small enough so that $\mathbb{B}(\bar{x},\gamma)\subset U$.
By the continuity of $\phi$ and the compactness of $\mathbb{B}(\bar{x},\gamma)$,
there is some $\gamma^{\prime}>0$ such that $\mathbb{B}(\bar{x},\gamma)\times[0,\gamma^{\prime}]\subset\phi^{-1}(U)$.

Next, suppose condition $(D)$ holds. Let $\gamma,\sigma,\kappa>0$
and $\phi:\mathbb{B}(\bar{x},\gamma)\times[0,\gamma]\rightarrow X$
satisfy the properties given in Definition \ref{def:Deformation-critical}
on deformation regularity. We can assume $\gamma>0$ is small enough
and choose $\gamma^{\prime}>0$ so that $\mathbb{B}_{}(\bar{x},\gamma+\gamma^{\prime}\kappa)\subset U$.
The conditions on $\phi$ imply that $\phi\left(\mathbb{B}(\bar{x},\gamma)\times[0,\gamma^{\prime}]\right)\subset\mathbb{B}(\bar{x},\gamma+\gamma^{\prime}\kappa)\subset U$,
which in turn imply that $\mathbb{B}(\bar{x},\gamma)\times[0,\gamma^{\prime}]\subset\phi^{-1}(U)$.

Here is the next argument common to both conditions $(D)$ and $(M_{\mathbb{R}^{n}})$.
By the characterization of saddle points in Proposition \ref{pro:equiv-mountain},
we can find $\theta$ and $\delta$ satisfying the condition in Proposition
\ref{pro:equiv-mountain}(b) with $\theta,\delta\leq\min(\frac{1}{2}\gamma^{\prime}\sigma,\gamma)$.
This gives us $\mathbb{B}(\bar{x},\delta)\subset\mathbb{B}(\bar{x},\gamma)\subset U$
in particular. We can glean from the proof of Proposition \ref{pro:equiv-mountain}
that we can find two points $a_{\delta}\in A\cap\mathbb{B}(\bar{x},\delta)$
and $b_{\delta}\in B\cap\mathbb{B}(\bar{x},\delta)$ and a path $p^{\prime}:[0,1]\rightarrow X$
connecting $a_{\delta}$ and $b_{\delta}$  contained in $\mathbb{B}(\bar{x},\delta)$
with maximum value at most $f(\bar{x})+\min(\frac{1}{2}\gamma^{\prime}\sigma,\gamma)$.
The functions values $f(a_{\delta})$ and $f(b_{\delta})$ satisfy
$f(a_{\delta}),f(b_{\delta})\leq f(\bar{x})-\theta$. The condition
$\mathbb{B}(\bar{x},\gamma)\times[0,\gamma^{\prime}]\subset\phi^{-1}(U)$
implies that $p^{\prime}([0,1])\times[0,\gamma^{\prime}]\subset\phi^{-1}(U)$.

If condition $(M_{b^{\prime}})$ holds, then for any $\delta>0$,
we can find a path $p^{\prime}:[0,1]\rightarrow X$ connecting two
points $a_{\delta}\in A\cap\mathbb{B}(\bar{x},\delta)$ and $b_{\delta}\in B\cap\mathbb{B}(\bar{x},\delta)$
contained in $\mathbb{B}(\bar{x},\delta)$ with maximum value at most
$f(\bar{x})$. There is also some $\theta>0$ such that $f(a_{\delta}),f(b_{\delta})<f(\bar{x})-\theta$.
Let $\gamma,\sigma>0$ and $\phi:\mathbb{B}(\bar{x},\gamma)\times[0,\gamma]\rightarrow X$
be such that they satisfy the properties of Morse regularity. By the
compactness of $p^{\prime}([0,1])$, we can find some $\gamma^{\prime}>0$
such that $p^{\prime}([0,1])\times[0,\gamma^{\prime}]\subset\phi^{-1}(U)$. 

To conclude the proof for all three cases, consider the path $\bar{p}:[0,3]\rightarrow X$
defined by\[
\bar{p}(t)=\begin{cases}
\phi(a_{\delta},\gamma^{\prime}t) & \mbox{ for }0\leq t\leq1\\
\phi(p^{\prime}(t-1),\gamma^{\prime}) & \mbox{ for }1\leq t\leq2\\
\phi(b_{\delta},\gamma^{\prime}(3-t)) & \mbox{ for }2\leq t\leq3.\end{cases}\]
This path connects $a_{\delta}$ and $b_{\delta}$, is contained in
$U$ and has maximum value at most $\max(f(\bar{x})-\theta,f(\bar{x})-\frac{1}{2}\gamma^{\prime}\sigma)$,
which is less than $f(\bar{x})$. This implies that $A$ and $B$
cannot be distinct path connected components of $(\lev_{<f(\bar{x})}f)\cap U$,
which establishes the contradiction in all three cases.
\end{proof}
We now move on to discuss how saddle points and deformationally
critical points relate to nonsmooth critical points. Here is the definition
of Clarke critical points.
\begin{defn}
\cite[Section 2.1]{Cla83} Let $X$ be a Banach space. Suppose $f:X\rightarrow\mathbb{R}$
is locally Lipschitz. The \emph{Clarke generalized directional derivative}
of $f$ at $x$ in the direction $v\in X$ is defined by\[
f^{\circ}(x;v)=\limsup_{t\searrow0,y\rightarrow x}\frac{f(y+tv)-f(y)}{t},\]
where $y\in X$ and $t$ is a positive scalar. The \emph{Clarke subdifferential}
of $f$ at $x$, denoted by $\partial_{C}f(x)$, is the convex subset
of the dual space $X^{*}$ given by\[
\{\zeta\in X^{*}\mid f^{\circ}(x;v)\geq\left\langle \zeta,v\right\rangle \mbox{ for all }v\in X\}.\]
The point $x$ is a \emph{Clarke (nonsmooth) critical point} if $\mathbf{0}\in\partial_{C}f(x)$.
Here, $\left\langle \cdot,\cdot\right\rangle :X^{*}\times X\rightarrow\mathbb{R}$
defined by $\left\langle \zeta,v\right\rangle :=\zeta(v)$ is the
dual relation.
\end{defn}
For the particular case of $\mathcal{C}^{1}$ functions, $\partial_{C}f(x)=\{\nabla f(x)\}$.
Therefore a critical point of a smooth function (i.e., a point $x$
that satisfies $\nabla f(x)=\mathbf{0}$) is also a Clarke critical
point. From the definitions above, it is clear that an equivalent
definition of a Clarke critical point is $f^{\circ}(x;v)\geq0$ for
all $v\in X$. This property allows us to deduce Clarke criticality
without appealing to the dual space $X^{*}$.

Clarke (nonsmooth) critical points of $f$ are of interest in, for
example, partial differential equations with discontinuous nonlinearities.
Critical point existence theorems for nonsmooth functions first appeared
in \cite{Chang81,Shi85}. For the problem of finding nonsmooth critical
points numerically, we are only aware of \cite{YZ05}. 

The following result is well-known, and we include its proof for completeness.
\begin{prop}
Let $X$ be a Banach space and $f:X\rightarrow\mathbb{R}$ be locally
Lipschitz at $\bar{x}$. If $\bar{x}$ is deformationally critical,
then it is Clarke critical.\end{prop}
\begin{proof}
We prove the contrapositive instead. If the point $\bar{x}$ is not
Clarke critical, there exists a unit vector $v\in X$ such that \[
\limsup_{t\searrow0,y\rightarrow\bar{x}}\frac{f(y+tv)-f(y)}{t}<0.\]
Now defining $\phi(x,t)=x-tv$ satisfies the conditions for deformation
regularity.
\end{proof}
To conclude, Figure \ref{fig:crit-equiv} summarizes the relationship
between saddle points and the different types of critical points.

\begin{figure}
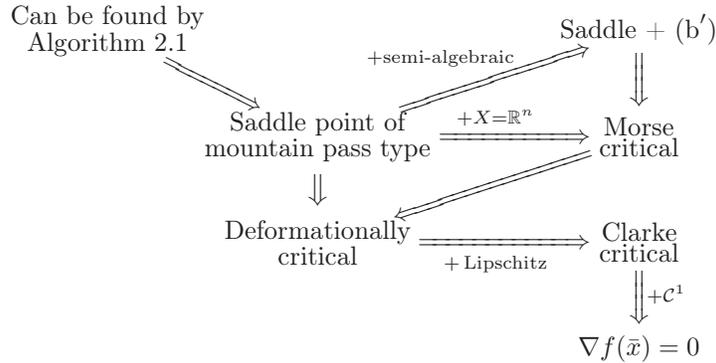

\begin{diagram} 
\substack{\mbox{Can be found by} \\\mbox{Algorithm \ref{alg:globally-convergent-MPT}}} & &  & &     \mbox{Saddle + (b}^\prime\mbox{)} \\ 
& \rdImplies &  & \ruImplies^{+\scriptsize{\mbox{semi-algebraic}}} & \dImplies  \\
& & 
\substack{\mbox{Saddle point of}\\\mbox{mountain pass type}} 
&\rImplies^{+X=\mathbb{R}^n}
&\substack{\mbox{Morse}\\\mbox{critical}}& & \\
& & \big\Downarrow &\ldImplies & & & \\
& & \substack{\mbox{Deformationally}\\\mbox{critical}}
&\rImplies_{+\mathop{\rm{Lipschitz}}}
&\substack{\mbox{Clarke}\\\mbox{critical}} \\
& & & & \dImplies_{+\mathcal{C}^1} \\
& & & & \nabla f(\bar{x})=0
\end{diagram}

\caption{\label{fig:crit-equiv}Different types of critical points}

\end{figure}

\section{\label{sec:Wilkinson-1}Wilkinson's problem: Background}

In Section \ref{sec:Wilkinson-2}, we will apply Algorithm \ref{alg:(Mountain-pass-1)}
to attempt to solve the Wilkinson problem, while we give a background
of the Wilkinson problem in this section. We first define the Wilkinson
problem.
\begin{defn}
Given a matrix $A\in\mathbb{R}^{n\times n}$, the \emph{Wilkinson
distance }of the matrix $A$ is the distance of the matrix $A$ to
the nearest matrix with repeated eigenvalues. The problem of finding
the Wilkinson distance is the \emph{Wilkinson problem}.
\end{defn}
Though not cited explicitly, as noted by \cite{AB05a}, the Wilkinson
problem can be traced back to \cite[pp. 90-93]{Wilk65}. See \cite{ABBO09,key-21,M99}
for more references, and in particular, \cite{ABBO09} and the discussion
in the beginning of \cite[Section 3]{key-21}. 

It is well-known that eigenvalues vary in a Lipschitz manner if and
only if they do not coincide. In fact, eigenvalues are differentiable
in the entries of the matrix when they are distinct. Hence, as discussed
by Demmel \cite{Dem87a}, the Wilkinson distance is a natural condition
measure for accurate eigenvalue computation. The Wilkinson distance
is also important because of its connections with the stability of
eigendecompositions of matrices. To our knowledge, no fast and reliable
numerical method for computing the Wilkinson distance is known. 

 The\emph{ $\epsilon$-pseudospectrum} $\Lambda_{\epsilon}(A)\subset\mathbb{C}$
of $A$ is defined as the set\begin{eqnarray*}
\Lambda_{\epsilon}(A) & := & \left\{ z\mid\exists E\mbox{ s.t. }\Vert E\Vert\leq\epsilon\mbox{ and }z\mbox{ is an eigenvalue of }A+E\right\} \\
 & = & \left\{ z\mid\left|(A-zI)^{-1}\right|^{-1}\leq\epsilon\right\} \\
 & = & \left\{ z\mid\underline{\sigma}(A-zI)\leq\epsilon\right\} ,\end{eqnarray*}
where $\underline{\sigma}(A-zI)$ is the smallest singular value of
$A-zI$. The function $z\mapsto(A-zI)^{-1}$ is sometimes referred
to as the resolvent function, whose (Clarke) critical points are referred
to as \emph{resolvent critical points}. To simplify notation, define
$\underline{\sigma}_{A}:\mathbb{C}\rightarrow\mathbb{R}_{+}$ by \begin{eqnarray*}
\underline{\sigma}_{A}(z) & := & \underline{\sigma}(A-zI)\\
 & = & \mbox{smallest singular value of }(A-zI).\end{eqnarray*}
For more on pseudospectra, we refer the reader to \cite{TE06}.

It is well known that each component of the $\epsilon$-pseudospectrum
$\Lambda_{\epsilon}(A)$ contains at least one eigenvalue. If $\epsilon$
is small enough, $\Lambda_{\epsilon}(A)$ has $n$ components, each
containing an eigenvalue. Alam and Bora \cite{AB05a} proved the following
result on the Wilkinson distance.
\begin{thm}
\cite{AB05a} Let $\bar{\epsilon}$ be the smallest $\epsilon$ for
which $\Lambda_{\epsilon}(A)$ contains $n-1$ or fewer components.
Then $\bar{\epsilon}$ is the Wilkinson distance for $A$. 

For any pair of distinct eigenvalues of $A$, say $\{z_{1},z_{2}\}$,
let the objective of the mountain pass problem with function $\underline{\sigma}_{A}$
and the two chosen eigenvalues as endpoints be $v(z_{1},z_{2})$.
The value $\bar{\epsilon}$ is also equal to \begin{equation}
\min\{v(z_{1},z_{2})\mid z_{1}\mbox{ and }z_{2}\mbox{ are distinct eigenvalues of }A\}.\label{eq:Wilk-min}\end{equation}

\end{thm}
Two components of $\Lambda_{\epsilon}(A)$ would coalesce when $\epsilon\uparrow\bar{\epsilon}$,
and the point at which two components coalesce can be used to construct
the matrix closest to $A$ with repeated eigenvalues. Equivalently,
the point of coalescence of the two components is also the highest
point on an optimal mountain pass for the function $\underline{\sigma}_{A}$
between the corresponding eigenvalues. We use Algorithm \ref{alg:(Mountain-pass-1)}
to find such points of coalescence, which are resolvent critical points.

We should remark that solving for $v(z_{1},z_{2})$ is equivalent
to solving a global mountain pass problem, which is difficult. Also,
the problem of finding the eigenvalue pair $\{z_{1},z_{2}\}$ that
minimizes \eqref{eq:Wilk-min} is potentially difficult. In Section
\ref{sec:Wilkinson-2}, we focus only on finding a critical point
of mountain pass type between two chosen eigenvalues $z_{1}$ and
$z_{2}$. Fortunately, this strategy often succeeds in obtaining the
Wilkinson distance in our experiments in Section \ref{sec:Wilkinson-2}.

We should note that other approaches for the Wilkinson problem include
\cite{ABBO09}, which uses a Newton type method for the same local
problem, and \cite{Mengi09}.

\section{\label{sec:Wilkinson-2}Wilkinson's problem: Implementation and numerical
results}

We first use a convenient fast heuristic to estimate which pseudospectral
components first coalesce as $\epsilon$ increases from zero, as follows.
We construct the Voronoi diagram corresponding to the spectrum, and
then minimize the function $\underline{\sigma}_{A}:\mathbb{C}\rightarrow\mathbb{R}$
over all the line segments in the diagram (a fast computation, as
discussed in the comments on Step 1(b) below). We then concentrate
on the pair of eigenvalues separated by the line segment containing
the minimizer. This is illustrated in Example \ref{exa:matrix} below.

We describe implementation issues of Algorithm \ref{alg:(Mountain-pass-1)}. 

\textbf{Step 1(a):} Approximately minimizing the distance between
a pair of points in distinct components seem challenging in practice,
as we discussed briefly in Section \ref{sec:locally-superlinearly-convergent}.
In the case of pseudospectral components, we have the advantage that
computing the intersection between any circle and the pseudospectral
boundary is an easy eigenvalue computation \cite{MO05}. This observation
can be used to to check optimality conditions or algorithm design
for step 1(a).  We note that in our numerical implementation, step
1(a) is never actually performed.

\textbf{Step 1(b):} Finding the global minimizer in step 1(b) of Algorithm
\ref{alg:(Mountain-pass-1)} is easy in this case. Byers \cite{Bye88}
proved that $\epsilon$ is a singular value of $A-(x+iy)I$ if and
only if $iy$ is an eigenvalue of \[
\left(\begin{array}{cc}
x-A^{*} & -\epsilon I\\
\epsilon I & A-x\end{array}\right).\]
Using Byer's observation, Boyd and Balakrishnan \cite{BB90} devised
a globally convergent and locally quadratic convergent method for
the minimization problem over $\mathbb{R}$ of $y\mapsto\underline{\sigma}_{A}(x+iy)$.
We can easily amend these observations to calculate the minimum of
$\underline{\sigma}_{A}(x+iy)$ over a line segment efficiently by
noticing that if $\left|z\right|=1$, then \[
\underline{\sigma}_{A}(x+iy)=\underline{\sigma}(A-(x+iy)I)=\underline{\sigma}(z(A-(x+iy)I)).\]

\begin{example}
\label{exa:matrix}We apply our mountain pass algorithm on the matrix
\[
A=\left(\begin{array}{ccccc}
.461+.650i & .006+.625i\\
 & .457+.983i & .297+.733i\\
 &  & .451+.553i & .049+.376i\\
 &  &  & .412+.400i & .693+.010i\\
 &  &  &  & .902+.199i\end{array}\right)\]

The results of the numerical algorithm are presented in Table \ref{tab:Convergence-data},
and plots using EigTooL \cite{key-14} are presented in Figure \ref{fig:4-pics}.
We tried many random examples of bidiagonal matrices taking entries
in the square $\{x+iy\mid0\leq x,y\leq1\}$ of the same form as $A$.
The convergence to a critical point in this example is representative
of the typical behavior we encountered.

In Figure \ref{fig:4-pics}, the top left picture shows that the first
step in the Voronoi diagram method identifies the pseudospectral components
corresponding to the eigenvalues $0.461+0.650i$ and $0.451+0.553i$
as the ones that possibly coalesce first. We zoom into these eigenvalues
in the top right picture. In the bottom left diagram, successive steps
in the bisection method gives better approximation of the saddle point.
Finally in the bottom right picture, we see that the saddle point
was calculated at an accuracy at which the level sets of $\underline{\sigma}_{A}$
are hard to compute.
\end{example}
There are other cases where the heuristic method fails to find the
correct pair of eigenvalues whose components first coalesce.
\begin{example}
\label{exa:treat-bad-cases}Consider the matrix $A$ generated by
the following Matlab code:

\begin{verbatim}
A=zeros(10); 

A(1:9,2:10)= diag([0.5330 + 0.5330i, 0.9370 + 0.1190i,...
 0.7410 + 0.8340i, 0.7480 + 0.8870i, 0.6880 + 0.6700i,...
 0.2510 + 0.7430i, 0.9540 + 0.6590i, 0.2680 + 0.6610i,...
 0.2670 + 0.4340i]);

A=       A+diag([0.9850 + 0.7550i,0.8030 + 0.7810i,...
0.2590 + 0.5110i,0.3840 + 0.5310i,0.0080 + 0.5360i,...
0.9780 + 0.2720i,0.7190 + 0.3100i,0.5560 + 0.8370i,...
0.6350 + 0.7630i,0.5110 + 0.8870i]);

\end{verbatim}A sample run for this matrix is shown in Figure \ref{fig:bad-case}.
The heuristic on minimal values of $\underline{\sigma}_{A}$ on the
edges of the Voronoi diagram identifies the top left and central eigenvalues
as a pair for which the pseudospectral components first coalesce.
However, the correct pair should be the central and bottom right eigenvalues.
\end{example}
Here are a few more observations. In our trials, we attempt to find
the Wilkinson distance for bidiagonal matrices of size $10\times10$
similar to the matrices in Examples \ref{exa:matrix} and \ref{exa:treat-bad-cases}.
In all the examples we have tried, there was no need to perform step
1(a) of Algorithm \ref{alg:(Mountain-pass-1)} to achieve convergence
to a critical point. The convergence for the matrix in Example \ref{exa:matrix}
reflects the general performance of the (local) algorithm. As we
have seen in Example \ref{exa:treat-bad-cases}, the heuristic for
choosing a pair of eigenvalues may fail to choose the correct pseudospectral
components which first coalesce as $\epsilon$ increases. In a sample
of 225 runs, we need to check other pairs of eigenvalues 7 times.
In such cases, a different choice of a pair of eigenvalues still gave
convergence to the Wilkinson distance, though whether this must always
be the case is uncertain. The upper bounds for the critical value
are also better approximates of the critical values than the lower
bounds.\setcounter{table}{0}

\begin{table}
\begin{tabular}{|c|l|l|l|l|}
\hline 
$i$ & $f(x_{i})$ & $M_{i}$ & $\frac{M_{i}-f(x_{i})}{f(x_{i})}$ & $\left|x_{i}-y_{i}\right|$\tabularnewline
\hline
\hline 
1 & \textbf{6.1}325135002707E-4 & \textbf{6.151109286}4335E-4 & 3.03E-03 & 5.23E-03\tabularnewline
\hline 
2 & \textbf{6.151109}1521293E-4 & \textbf{6.151109286142}6E-4 & 2.18E-08 & 1.40E-05\tabularnewline
\hline 
3 & \textbf{6.151109286142}2E-4 & \textbf{6.151109286142}3E-4 & 3.35E-15 & 9.97E-10\tabularnewline
\hline
\end{tabular}\caption{\label{tab:Convergence-data}Convergence data for Example \ref{exa:matrix}.
Significant digits are in bold.}

\end{table}

\begin{figure}

\includegraphics[scale=0.5]{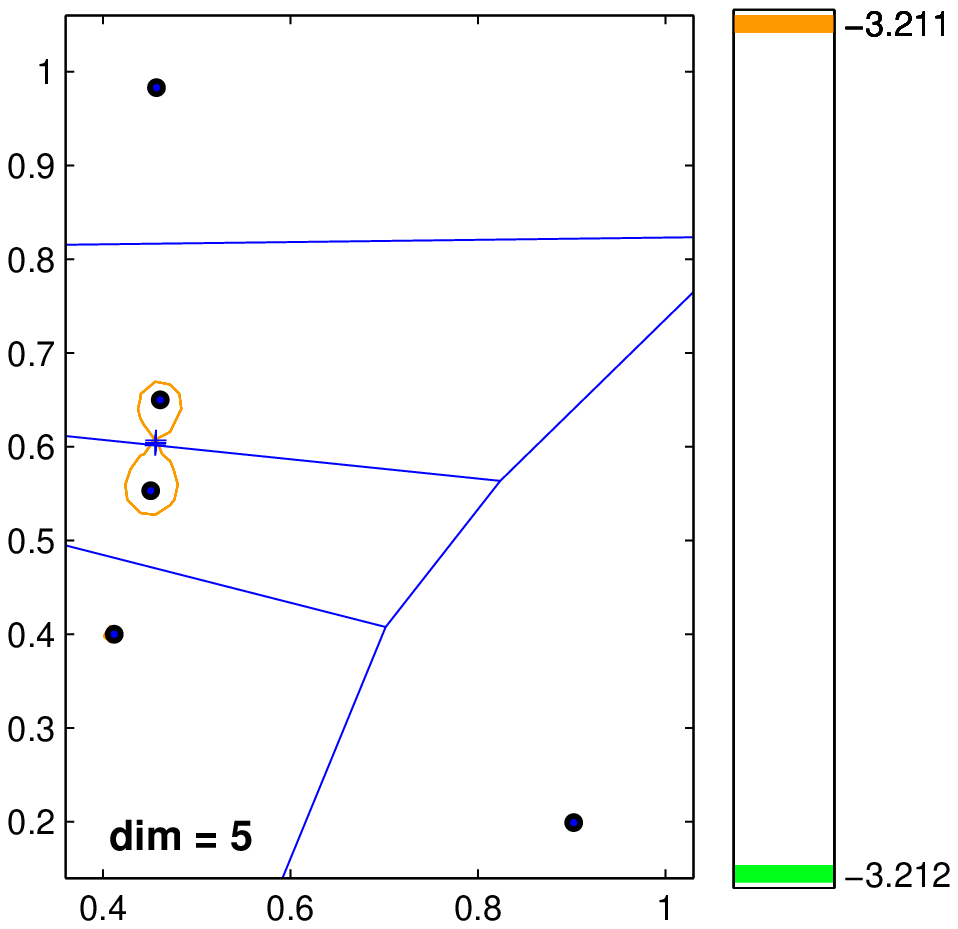}\includegraphics[scale=0.5]{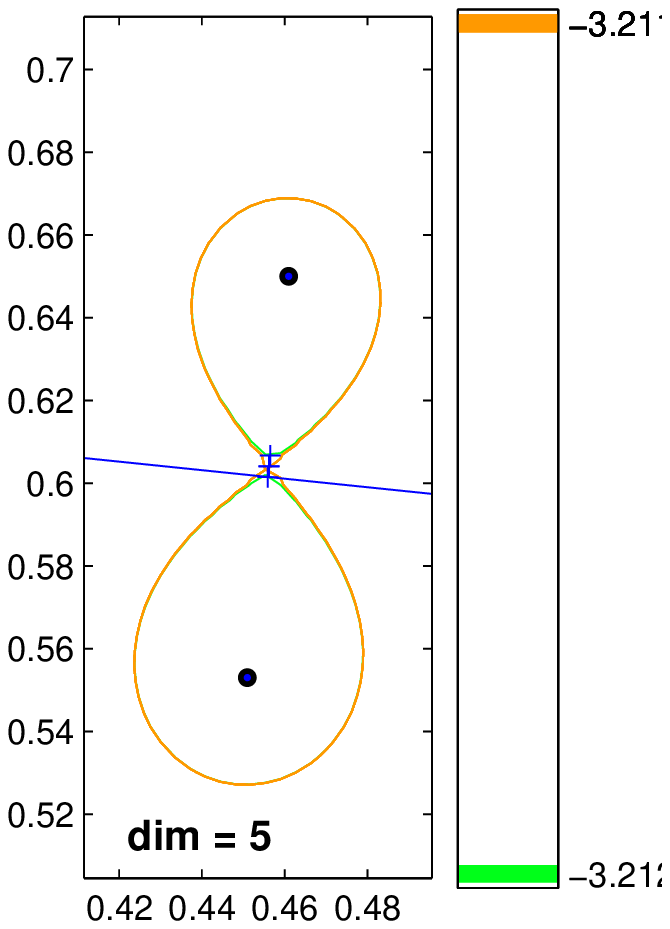}

\includegraphics[scale=0.5]{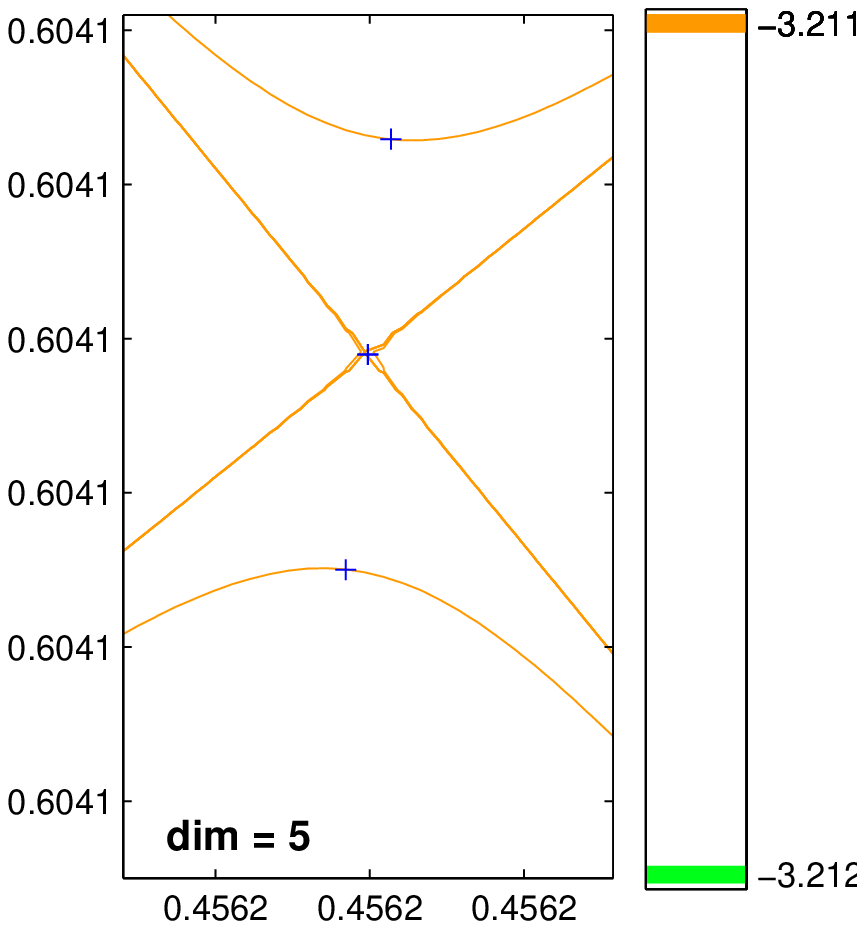}\includegraphics[scale=0.5]{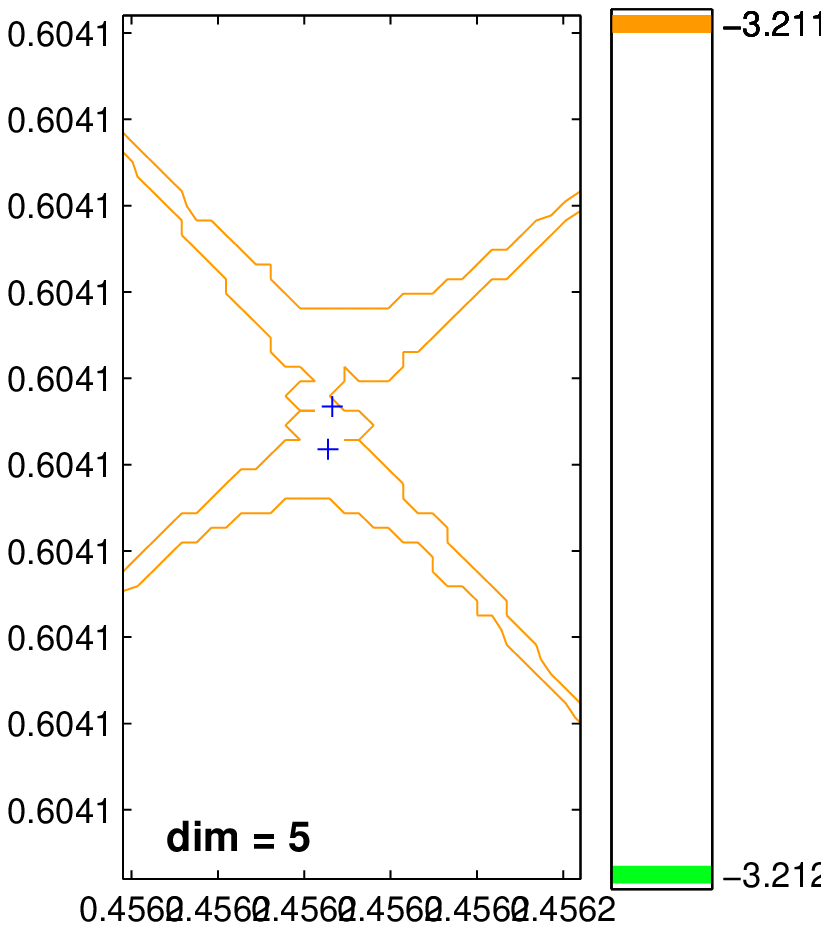}\caption{\label{fig:4-pics}A sample run of Algorithm \ref{alg:(Mountain-pass-1)}.}

\end{figure}

\begin{figure}

\includegraphics[scale=0.5]{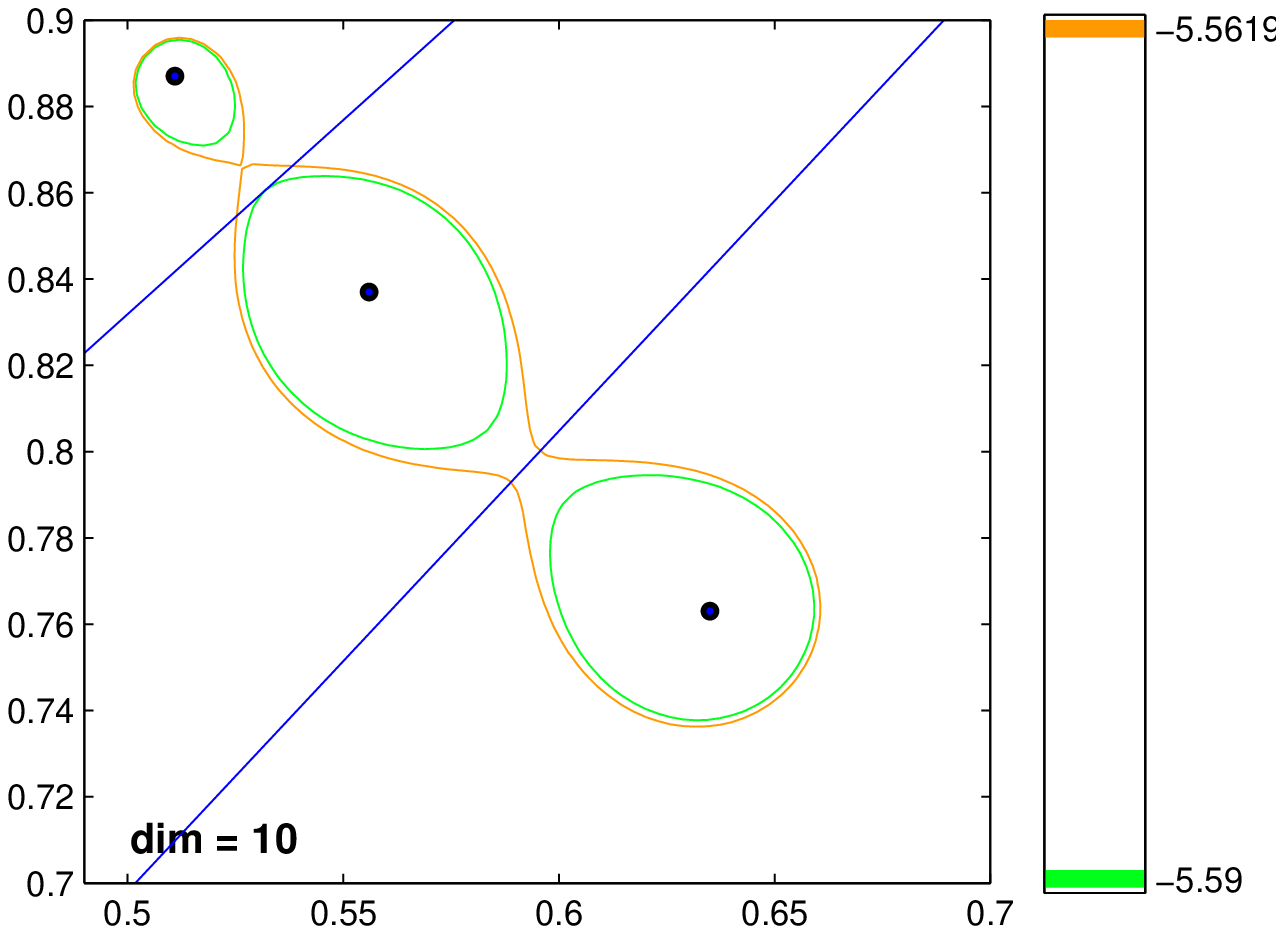}\includegraphics[scale=0.5]{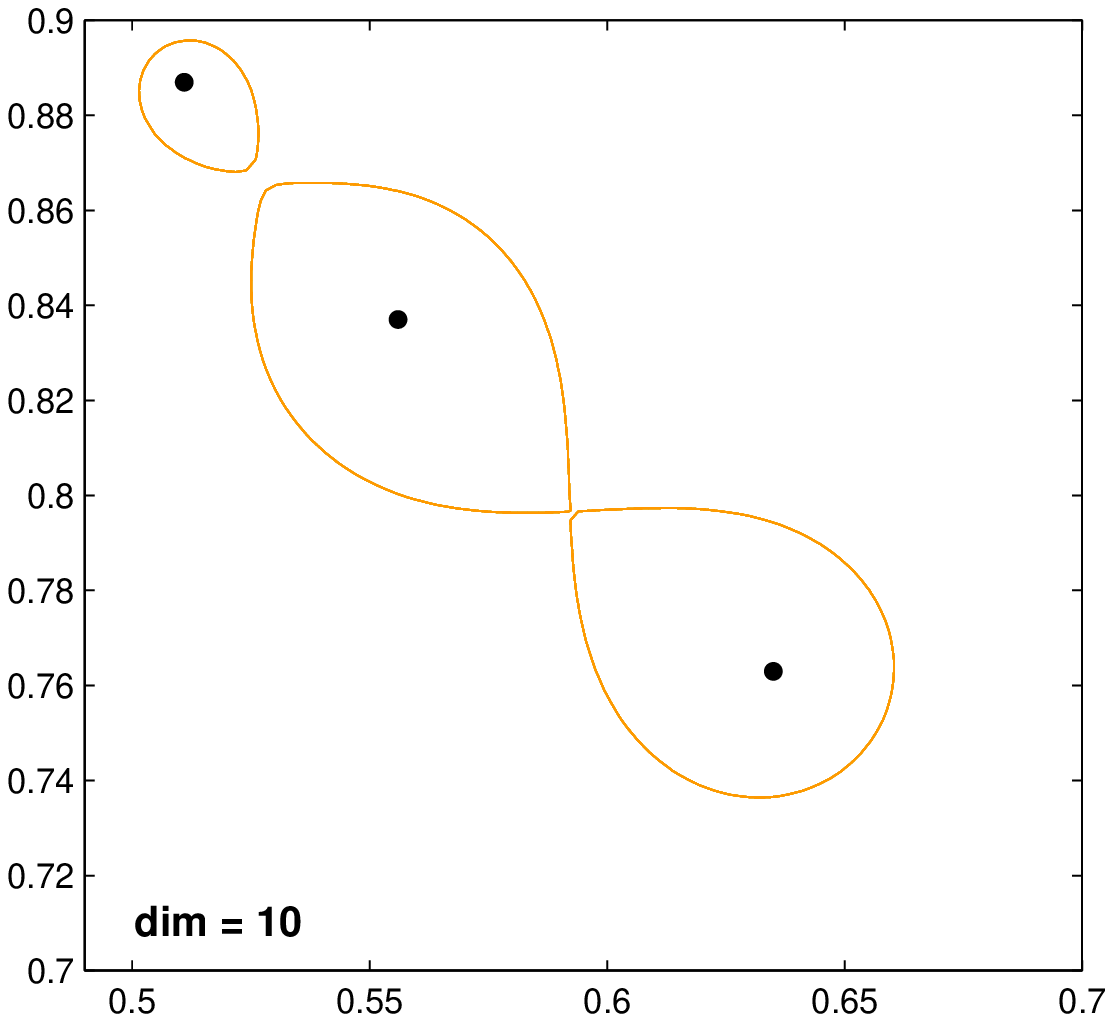}

\caption{\label{fig:bad-case}An example where the Voronoi diagram heuristic
fails.}

\end{figure}

\section{\label{sec:optim-conditions}Non-Lipschitz convergence and optimality
conditions}

In this section, we discuss the convergence of Algorithm \ref{alg:globally-convergent-MPT}
in the non-Lipschitz case and give an optimality condition in step
2 of Algorithm \ref{alg:globally-convergent-MPT}. As one might expect
in the smooth case in a Hilbert space, if $x_{i}$ and $y_{i}$ are
closest points in the different components, $\nabla f(x_{i})\neq\mathbf{0}$
and $\nabla f(y_{i})\neq\mathbf{0}$, then we have \[
\begin{array}{r}
x_{i}-y_{i}=\lambda_{1}\nabla f(y_{i}),\\
y_{i}-x_{i}=\lambda_{2}\nabla f(x_{i}).\end{array}\]
for $\lambda_{1},\lambda_{2}>0$. The rest of this section extends
this result to the nonsmooth case, making use of the language of variational
analysis in the style of \cite{RW98,BZ05,Cla83,Mor06} to describe
the relation between subdifferentials of $f$ and the normal cones
of the level sets of $f$. 

We now recall the definition of the Fr\'{e}chet subdifferential,
which is a generalization of the derivative to nonsmooth cases, and
the Fr\'{e}chet normal cone. A function $f:X\rightarrow\mathbb{R}$
is \emph{lsc} (lower semicontinuous) if $\liminf_{x\rightarrow\bar{x}}f(x)\geq f(\bar{x})$
for all $\bar{x}\in X$.
\begin{defn}
Let $f:X\rightarrow\mathbb{R}\cup\left\{ +\infty\right\} $ be a
proper lsc function. We say that $f$ is \emph{Fr\'{e}chet subdifferentiable}
and $x^{*}$ is a \emph{Fr\'{e}chet-subderivative} of $f$ at $x$
if $x\in\dom f$ and \[
\liminf_{\left|h\right|\rightarrow0}\frac{f(x+h)-f(x)-\left\langle x^{*},h\right\rangle }{\left|h\right|}\geq0.\]
 We denote the set of all Fr\'{e}chet-subderivatives of $f$ at $x$
by $\partial_{F}f(x)$ and call this object the \emph{Fr\'{e}chet
subdifferential} of $f$ at $x$.
\begin{defn}
Let $S$ be a closed subset of $X$. We define the \emph{Fr\'{e}chet
normal cone} of $S$ at $x$ to be $N_{F}(S;x):=\partial_{F}\iota_{S}(x)$.
Here, $\iota_{S}:X\rightarrow\mathbb{R}\cup\left\{ \infty\right\} $
is the indicator function defined by $\iota_{S}(x)=0$ if $x\in S$,
and $\infty$ otherwise.
\end{defn}
\end{defn}
Closely related to the Fr\'{e}chet normal cone is the proximal normal
cone.
\begin{defn}
Let $X$ be a Hilbert space and let $S\subset X$ be a closed set.
If $x\notin S$ and $s\in S$ are such that $s$ is a closest point
to $x$ in $S$, then any nonnegative multiple of $x-s$ is a \emph{proximal
normal vector} to $S$ at $s$. The set of all proximal normal vectors
is denoted $N_{P}(S;s)$.
\end{defn}
The proximal normal cone and the Fr\'{e}chet normal cone satisfy
the following relation. See for example \cite[Exercise 5.3.5]{BZ05}.
\begin{thm}
$N_{P}(S;\bar{x})\subset N_{F}(S;\bar{x})$.
\end{thm}
Here is an easy consequence of the definitions.
\begin{prop}
\label{pro:opp-normals} Let $S_{1}$ be the component of $\lev_{\leq l_{i}}f$
containing $x_{0}$ and $S_{2}$ be the component of $\lev_{\leq l_{i}}f$
containing $y_{0}$. Suppose that $x_{i}$ is a point in $S_{1}$
closest to $S_{2}$ and $y_{i}$ is a point in $S_{2}$ closest to
$x_{i}$. Then we have \[
(y_{i}-x_{i})\in N_{P}(\lev_{\leq l_{i}}f;x_{i})\subset N_{F}(\lev_{\leq l_{i}}f;x_{i}).\]
Similarly, $(x_{i}-y_{i})\in N_{F}(\lev_{\leq l_{i}}f;y_{i})$. These
are two normals of $\lev_{\leq l_{i}}f$ pointing in opposite directions. 
\end{prop}
The above result gives a necessary condition for the optimality of
step 2 in Algorithm \ref{alg:globally-convergent-MPT}. We now see
how the Fr\'{e}chet normals relate to the subdifferential of $f$
at $x_{i}$, $y_{i}$ at $\bar{z}$. Here is the definition of the
Clarke subdifferential for non-Lipschitz functions.
\begin{defn}
 Let $X$ be a Hilbert space and let $f:X\rightarrow\mathbb{R}$
be a lsc function. Then the \emph{Clarke subdifferential} of $f$
at $\bar{x}$ is \[
\partial_{C}f(\bar{x}):=\cl\,\conv\{\wlim_{i\rightarrow\infty}x_{i}^{*}\mid x_{i}^{*}\in\partial_{F}f(x_{i}),(x_{i},f(x_{i}))\rightarrow(\bar{x},f(\bar{x}))\}+\partial_{C}^{\infty}f(\bar{x}),\]
where the \emph{singular subdifferential} of $f$ at $\bar{x}$ is
a cone defined by \[
\partial_{C}^{\infty}f(\bar{x}):=\cl\,\conv\{\wlim_{i\rightarrow\infty}\lambda_{i}x_{i}^{*}\mid x_{i}^{*}\in\partial_{F}f(x_{i}),(x_{i},f(x_{i}))\rightarrow(\bar{x},f(\bar{x})),\lambda_{i}\rightarrow0_{+}\}.\]

\end{defn}
For finite dimensional spaces, the weak topology is equivalent to
the norm topology, so we may replace $\wlim$ by $\lim$ in that setting.
We will use the limiting subdifferential and the limiting normal cone,
whose definitions we recall below, in the proof of the finite dimensional
case of Theorem \ref{thm:second-critical-pt-thm}.
\begin{defn}
Let $X$ be a Hilbert space and let $f:X\rightarrow\mathbb{R}$ be
a lsc function. Define the \emph{limiting subdifferential} of $f$
at $\bar{x}$ by \[
\partial_{L}f(\bar{x}):=\{\wlim_{i\rightarrow\infty}x_{i}^{*}\mid x_{i}^{*}\in\partial_{F}f(x_{i}),\left(x_{i},f(x_{i})\right)\rightarrow\left(\bar{x},f(\bar{x})\right)\},\]
and the \emph{singular subdifferential} of $f$ at $\bar{x}$, which
is a cone, by \[
\partial^{\infty}f(\bar{x}):=\{\wlim_{i\rightarrow\infty}t_{i}x_{i}^{*}\mid x_{i}^{*}\in\partial_{F}f(x_{i}),(x_{i},f(x_{i}))\rightarrow(\bar{x},f(\bar{x})),t_{i}\rightarrow0_{+}\}.\]

\end{defn}
The limiting normal cone is defined in a similar manner.
\begin{defn}
Let $X$ be a Hilbert space and let $S$ be a closed subset of $X$.
Define the limiting normal cone of $S$ at $x$ by\[
N_{L}(S;x):=\{\wlim_{i\rightarrow\infty}x_{i}^{*}\mid x_{i}^{*}\in N_{F}(S;x_{i}),S\ni x_{i}\rightarrow x\}.\]

\end{defn}
It is clear from the definitions that the Fr\'{e}chet subdifferential
is contained in the limiting subdifferential, which is in turn contained
in the Clarke subdifferential. Similarly, the Fr\'{e}chet normal
cone is contained in the limiting normal cone. We first state a theorem
relating normal cones to subdifferentials in the finite dimensional
case.
\begin{thm}
\label{thm:[RW98,10.3]}\cite[Proposition 10.3]{RW98} For a lsc function
$f:\mathbb{R}^{n}\rightarrow\mathbb{R}\cup\left\{ \infty\right\} $,
let $\bar{x}$ be a point with $f(\bar{x})=\alpha$. Then\[
N_{F}(\lev_{\leq\alpha}f;\bar{x})\supset\mathbb{R}_{+}\partial_{F}f(\bar{x})\cup\left\{ \mathbf{0}\right\} .\]
If $\partial_{L}f(\bar{x})\not\ni\mathbf{0}$, then also\[
N_{L}(\lev_{\leq\alpha}f;\bar{x})\subset\mathbb{R}_{+}\partial_{L}f(\bar{x})\cup\partial^{\infty}f(\bar{x}).\]

\end{thm}
The corresponding result for the infinite dimensional case is presented
below.
\begin{thm}
\label{thm:[BZ05,3.3.4]}\cite[Theorem 3.3.4]{BZ05} Let $X$ be a
Hilbert space and let $f:X\rightarrow\mathbb{R}\cup\left\{ +\infty\right\} $
be a lsc function. Suppose that $\liminf_{x\rightarrow\bar{x}}d(\partial_{F}f(x);\mathbf{0})>0$
and $\xi\in N_{F}(\lev_{\leq f(\bar{x})}f;\bar{x})$. Then, for any
$\epsilon>0$, there exist $\lambda>0$, $(x,f(x))\in\mathbb{B}_{\epsilon}((\bar{x},f(\bar{x})))$
and $x^{*}\in\partial_{F}f(x)$ such that \[
\left|\lambda x^{*}-\xi\right|\leq\epsilon.\]

\end{thm}
With these preliminaries, we now prove our theorem for the convergence
of Algorithm \ref{alg:globally-convergent-MPT} to a Clarke critical
point.
\begin{thm}
\label{thm:second-critical-pt-thm}Suppose that $f:X\rightarrow\mathbb{R}$,
where $X$ is a Hilbert space and $f$ is lsc. If $\bar{z}$ is such
that
\begin{enumerate}
\item $\left(\bar{z},\bar{z}\right)$ is a limit point of $\left\{ (x_{i},y_{i})\right\} _{i=1}^{\infty}$
in Algorithm \ref{alg:globally-convergent-MPT}, and 
\item $f$ is continuous at $\bar{z}$.
\end{enumerate}
Then one of these must hold:
\begin{enumerate}
\item [(a)] $\bar{z}$ is a Clarke critical point, 
\item [(b)] $\partial_{C}^{\infty}f(\bar{z})$ contains a line through
the origin, or
\item [(c)] $\left\{ \frac{y_{i}-x_{i}}{\left|y_{i}-x_{i}\right|}\right\} _{i}$
converges weakly to zero.
\end{enumerate}
\end{thm}
\begin{proof}
 We present both the finite dimensional and infinite dimensional
versions of the proof to our result. 

Suppose the subsequence $\left\{ (x_{i},y_{i})\right\} _{i\in J}$
is such that $\lim_{i\rightarrow\infty,i\in J}(x_{i},y_{i})=(\bar{z},\bar{z})$,
where $J\subset\mathbb{N}$. We can choose $J$ so that none of the
elements in $\left\{ (x_{i},y_{i})\right\} _{i\in J}$ are such that
$\liminf_{x\rightarrow x_{i}}d\left(\partial_{F}f(x);\mathbf{0}\right)=0$
or $\liminf_{y\rightarrow y_{i}}d\left(\partial_{F}f(y);\mathbf{0}\right)=0$,
otherwise we have $\mathbf{0}\in\partial_{C}f(\bar{z})$ by the definition
of the Clarke subdifferential, which is what we seek to prove. (In
finite dimensions, the condition $\liminf_{x\rightarrow x_{i}}d(\partial_{F}f(x);\mathbf{0})=0$
can be replaced by $\mathbf{0}\in\partial_{L}f(x_{i})$.) We proceed
to apply Theorem \ref{thm:[BZ05,3.3.4]} (and Theorem \ref{thm:[RW98,10.3]}
for finite dimensions) to find out more about $N_{F}(\lev_{\leq l_{i}}f;x_{i})$.

We first prove the result for finite dimensions. If $\mathbf{0}\in\partial_{L}f(\bar{z})$,
we are done. Otherwise, by Proposition \ref{pro:opp-normals} and
Theorem \ref{thm:[RW98,10.3]}, there is a positive multiple of $v=\lim_{i\rightarrow\infty}\frac{y_{i}-x_{i}}{\left|y_{i}-x_{i}\right|}$
that lies in either $\partial_{L}f(\bar{z})$ or $\partial^{\infty}f(\bar{z})$.
Similarly, there is a positive multiple of $-v=\lim_{i\rightarrow\infty}\frac{x_{i}-y_{i}}{\left|y_{i}-x_{i}\right|}$
lying in either $\partial_{L}f(\bar{z})$ or $\partial^{\infty}f(\bar{z})$.
If either $v$ or $-v$ lies in $\partial_{L}f(\bar{z})$, then we
can conclude $\mathbf{0}\in\partial_{C}f(\bar{z})$ from the definitions.
Otherwise both $v$ and $-v$ lie in $\partial_{C}^{\infty}f(\bar{z})$,
so $\mathbb{R}\left\{ v\right\} \subset\partial_{C}^{\infty}f(\bar{z})$
as needed.

We now prove the result for infinite dimensions. The point $\bar{z}$
is the common limit of $\left\{ x_{i}\right\} _{i\in J}$ and $\left\{ y_{i}\right\} _{i\in J}$.
By the optimality of $\left|x_{i}-y_{i}\right|$ and Proposition \ref{pro:opp-normals},
we have $y_{i}-x_{i}\in N_{F}(\lev_{\leq l_{i}}f;x_{i})$ and $x_{i}-y_{i}\in N_{F}(\lev_{\leq l_{i}}f;y_{i})$.
By Theorem \ref{thm:[BZ05,3.3.4]}, for any $\kappa_{i}\rightarrow0_{+}$,
there is a $\lambda_{i}>0$, $x_{i}^{\prime}\in\mathbb{B}_{\kappa_{i}\left|x_{i}-y_{i}\right|}(x_{i})$
and $x_{i}^{*}\in\partial_{F}f\left(x_{i}^{\prime}\right)$ such that
$\left|\lambda_{i}x_{i}^{*}-(y_{i}-x_{i})\right|<\kappa_{i}\left|y_{i}-x_{i}\right|$.
Similarly, there is a $\gamma_{i}>0$, $y_{i}^{\prime}\in\mathbb{B}_{\kappa_{i}\left|y_{i}-x_{i}\right|}(y_{i})$
and $y\in\partial_{F}f(y_{i}^{\prime})$ such that $\left|\gamma_{i}y_{i}^{*}-(x_{i}-y_{i})\right|<\kappa_{i}\left|x_{i}-y_{i}\right|$.
If either $x_{i}^{*}$ or $y_{i}^{*}$ converges to $\mathbf{0}$,
then $\mathbf{0}\in\partial_{C}f(\bar{z})$, and we are done. Otherwise,
by the Banach Aloaglu theorem, the unit ball is compact, so $\left\{ \frac{1}{\left|x_{i}^{*}\right|}x_{i}^{*}\right\} _{i}$
and $\left\{ \frac{1}{\left|y_{i}-x_{i}\right|}(y_{i}-x_{i})\right\} _{i}$
have weak cluster points. We now show that they must have the same
cluster points by showing that their difference converges to $\mathbf{0}$
(in the strong topology). Now,\begin{eqnarray*}
\left|\frac{\lambda_{i}x_{i}^{*}}{\left|y_{i}-x_{i}\right|}\right| & \leq & \left|\frac{\lambda_{i}x_{i}^{*}}{\left|y_{i}-x_{i}\right|}-\frac{y_{i}-x_{i}}{\left|y_{i}-x_{i}\right|}\right|+\left|\frac{y_{i}-x_{i}}{\left|y_{i}-x_{i}\right|}\right|\\
 & \leq & \kappa_{i}+1,\end{eqnarray*}
and similarly, $1-\kappa_{i}\leq\left|\frac{\lambda_{i}x_{i}^{*}}{\left|y_{i}-x_{i}\right|}\right|$,
so $\left|\frac{\lambda_{i}x_{i}^{*}}{\left|y_{i}-x_{i}\right|}\right|\rightarrow1$,
and thus \[
\left|\frac{\lambda_{i}x_{i}^{*}}{\left|y_{i}-x_{i}\right|}-\frac{x_{i}^{*}}{\left|x_{i}^{*}\right|}\right|=\left|\left|\frac{\lambda_{i}x_{i}^{*}}{\left|y_{i}-x_{i}\right|}\right|-\left|\frac{x_{i}^{*}}{\left|x_{i}^{*}\right|}\right|\right|\rightarrow0.\]
This means that \[
\left|\frac{x_{i}^{*}}{\left|x_{i}^{*}\right|}-\frac{y_{i}-x_{i}}{\left|y_{i}-x_{i}\right|}\right|\leq\left|\frac{\lambda_{i}x_{i}^{*}}{\left|y_{i}-x_{i}\right|}-\frac{x_{i}^{*}}{\left|x_{i}^{*}\right|}\right|+\left|\frac{\lambda_{i}x_{i}^{*}}{\left|y_{i}-x_{i}\right|}-\frac{y_{i}-x_{i}}{\left|y_{i}-x_{i}\right|}\right|\rightarrow0,\]
which was what we claimed earlier. This implies that $\frac{x_{i}^{*}}{\left|x_{i}^{*}\right|}$
and $\frac{y_{i}^{*}}{\left|y_{i}^{*}\right|}$ have weak cluster
points that are the negative of each other. 

We now suppose that conclusion (c) does not hold. If $\left\{ x_{i}^{*}\right\} _{i}$
has a nonzero weak cluster point, say $\bar{x}^{*}$, then $\bar{x}^{*}$
belongs to $\partial_{C}f(\bar{z})$. Then $\left\{ y_{i}^{*}\right\} _{i}$
either has a weak cluster point $\bar{y}^{*}$ that is strictly a
negative multiple of $\bar{x}^{*}$, which implies that $\mathbf{0}\in\partial_{C}f(\bar{z})$
as claimed, or there is some $\bar{y}^{*,\infty}\in\partial_{C}^{\infty}f(\bar{z})$
which is a negative multiple of $\bar{x}^{*}$, which also implies
that $\mathbf{0}\in\partial_{C}f(\bar{z})$ as needed. 

If neither $\left\{ x_{i}^{*}\right\} _{i}$ or $\left\{ y_{i}^{*}\right\} _{i}$
converges weakly, then two (nonzero) weak cluster points of $\frac{x_{i}^{*}}{\left|x_{i}^{*}\right|}$
and $\frac{y_{i}^{*}}{\left|y_{i}^{*}\right|}$ that point in opposite
directions give a line through the origin in $\partial_{C}^{\infty}f(\bar{z})$
as needed.
\end{proof}
In finite dimensions, conclusion (b) of Theorem \ref{thm:second-critical-pt-thm}
is precisely the lack of {}``epi-Lipschitzness'' \cite[Exercise 9.42(b)]{RW98}
of $f$. One example where Algorithm \ref{alg:globally-convergent-MPT}
does not converge to a Clarke critical point but to a point with its
singular subdifferential $\partial_{C}^{\infty}f(\cdot)$ containing
a line through the origin is $f:\mathbb{R}\rightarrow\mathbb{R}$
defined by $f(x)=-\sqrt{\left|x\right|}$. Algorithm \ref{alg:globally-convergent-MPT}
converges to the point $0$, where $\partial_{C}f(0)=\emptyset$ and
$\partial_{C}^{\infty}f(0)=\mathbb{R}$. We do not know of an example
where only condition (c) holds.

\section*{Acknowledgments}

We thank Jianxin Zhou for comments on an earlier version of the manuscript,
and we thank an anonymous referee for feedback, which have improved
the presentation in the paper.

\end{document}